\documentclass[11pt]{amsart}
\usepackage{fullpage}
\usepackage{graphicx}
\usepackage{amssymb}
\usepackage{epstopdf}
\usepackage{color}
\usepackage{bbm, dsfont}
\RequirePackage[colorlinks,citecolor=blue,urlcolor=blue]{hyperref}

\newtheorem{theorem}{Theorem}
\newtheorem{lemma}{Lemma}[section]

\newtheorem{proposition}{Proposition}[section]

\theoremstyle{definition}

\theoremstyle{remark}
\newtheorem{remark}{Remark}
\theoremstyle{example}

\DeclareFontFamily{OML}{rsfs}{\skewchar\font'177}
\DeclareFontShape{OML}{rsfs}{m}{n}{ <5> <6> rsfs5 <7> <8> <9>
rsfs7 <10> <10.95> <12> <14.4> <17.28> <20.74> <24.88> rsfs10 }{}
\DeclareMathAlphabet{\mathfs}{OML}{rsfs}{m}{n}

\newcommand{\BN}{{\mathbb{N}}}

\newcommand{\BT}{{\mathbb{T}}}

\newcommand{\BZ}{{\mathbb{Z}}}

\newcommand{\CB}{{\mathcal{B}}}

\newcommand{\CG}{{\mathcal{G}}}

\newcommand{\CI}{{\mathcal{I}}}

\newcommand{\CN}{{\mathcal{N}}}

\newcommand{\ind}{{\mathbbm{1}}}

\newcommand{\e}{{\bf E}}
\newcommand{\bea}{\begin{equation}\begin{aligned}}
\newcommand{\eea}{\end{aligned}\end{equation}}

\newcommand{\ep}{{\epsilon}}

\numberwithin{equation}{section} 

\begin{document}
\title{on some threshold-one attractive interacting particle systems on homogeneous trees}

\author{Yingxin Mu}
\address[Yingxin Mu]{Peking University}
\email{muyingxin@pku.edu.cn}

\author{Yuan Zhang}
\address[Yuan Zhang]{Peking University}
\email{zhangyuan@math.pku.edu.cn}

\thanks{The authors would like to thank Dr. Rick Durrett, Dr. Thomas Liggett, and Dr. Thomas Mountford for helpful comments, and Dr. Eviatar B. Procaccia for allowing us to use his LaTeX environment settings and shortcuts}

\maketitle

\tableofcontents

\begin{abstract}
In this paper, we consider the threshold-one contact process and the threshold-one voter model w/o spontaneous death on homogeneous trees $\BT_d$, $d\ge 2$. Mainly inspired by the corresponding arguments for ordinary contact processes, we prove that the complete convergence theorem holds for these three systems under strong survival. When the systems survives weakly, complete convergence may also hold under certain transition and/or initial conditions. 
\end{abstract}

\section{Introduction}

\label{section_1}

In this paper, we study a threshold variant of contact process, the threshold-one contact process, and two threshold variants of voter model, the threshold-one voter model w/o spontaneous death, on homogeneous trees. The {\bf threshold contact processes} and the {\bf threshold voter model} were firstly defined by Cox and Durrett in \cite{cox1991nonlinear}. Recall that for the ordinary contact process and voter model, the interacting rate, i.e., the birth rate for contact process and the flipping rates for voter model, are linear functions with respect to the number of particles of different type in the neighborhood, denoted by $N$. In these variant models, these rates are replaced by a step function of $N$ which jumps from 0 to $\lambda$ at $N=\theta$. For the threshold contact process, when $\theta\ge 2$, the model has been studied in works including \cite{mountford2009survival}, when $\theta=1$, \cite{cox1991nonlinear} proves a number of properties for the process on $\BZ^d$, including the complete convergence theorem, and \cite{penrose1996threshold} studies the asymptotic behavior of the critical value for long interacting ranges. More recently, in \cite{xue2015asymptotic}, Xue revisits this case and call it the {\bf``threshold-one contact process"}. In \cite{xue2015asymptotic}, Xue shows that the critical birth/infection rate $\lambda$ for threshold-one contact processes on high dimensional lattice $\sim (2d)^{-1}$ where $d$ is the dimension. Later,  in \cite{xue2016convergence}, he shows the convergence rate when the system dies out. For the threshold voter model, \cite{mountford2009survival} raises several conjectures and proves that the complete convergence theorem holds for the threshold-one voter model with positive spontaneous death starting from the translation invariant measure. Later, a number of works studying the threshold voter models on $\BZ^d$ appeared, including \cite{cox1991nonlinear},  \cite{liggett1999stochastic},  \cite{andjel1992clustering}, and \cite{handjani1999complete}. Especially, when $\theta=1,(d,M)\ne(1,1)$, \cite{liggett1999stochastic}  proves the coexistence and \cite{handjani1999complete} proves the complete convergence theorem.

At the same time, although contact process and its variants were originally studied on lattice, they may also be defined on other finite or infinite graphs, including the $d-$dimensional homogeneous trees $\BT_d$. For the ordinary contact process on $\BT_d$, its behavior differs from the one on lattice in the sense that there exists a second phase transition of {\bf strong/local survival} which opens the possibility that distribution of the process may still go to zero, even if the population itself persists, see \cite{morrow1994critical,durrett1995intermediate,stacey1996existence,salzano1998new,liggett1999stochastic} for details. Works such as \cite{ziezold1988critical, pemantle1992contact, liggett1995improved, liggett1999stochastic} give upper and lower bounds estimates for the critical values.  When the ordinary contact process on $\BT_d$ survives strongly, the complete convergence theorem was shown in \cite{zhang1996complete}, while an alternative approach can be found in \cite{salzano1998new,liggett1999stochastic}. In \cite{stacey2001contact} and \cite{Cranston2014}, the process restricted on a finite subtree is also discussed. When $\theta=1$, for the threshold-one contact process, \cite{xue2015asymptotic} shows that the critical value $\sim d^{-1}$ for large $d$. However, until now, little is known to our knowledge about the threshold voter model on $\BT_d$. 

As discussed above, a particle system on homogeneous trees $\BT_d$ may exhibit a second phase transition of strong/local survival. To be precise, consider $\theta^{\lambda,\BT_d}_t$ to be a translation invariant interacting system on $\{0,1\}^{\BT_d}$ that is monotone with respect to $\lambda$ and attractive with respect to the natural order on $\{0,1\}^{\BT_d}$. Denote $\textbf{0}$ as the configuration such that $\eta(x)\equiv 0$ for all $x\in\BT_d$ and $\textbf{1}$ as the configuration such that $\eta(x)\equiv 1$ for all $x\in\BT_d$. Without loss of generality, we can assume that $\theta^{\lambda,\BT_d}_t$ is monotonically increasing with respect to $\lambda$. Then the different senses of survival/critical value(s) can be defined as follows, see \cite{durrett1995intermediate} for example: 

\begin{itemize} 
\item $\theta^{\lambda,\BT_d}_t$ {\bf survives from finite population } if for all $x$
$$
\lim_{t\to\infty} P^x\left(\theta^{\lambda,\BT_d}_t\not=\textbf{0} \right)>0,
$$
otherwise it {\bf dies out}.
We define critical value
$$
\lambda_e(d)=\inf\{\lambda: \ \text{$\theta^{\lambda,\BT_d}_t$ survives from finite population}\}. 
$$
\item $\theta^{\lambda,\BT_d}_t$ {\bf survives} if 
$$
\bar\nu^\lambda=\lim_{t\to\infty} \delta_\mathbf{1} S^\lambda(t)\not=\delta_\mathbf{0},
$$
where $S^\lambda(\cdot)$ stands for the semi-group of $\theta^{\lambda,\BT_d}_t$. We define critical value
$$
\lambda_1(d)=\inf\{\lambda: \ \text{$\theta^{\lambda,\BT_d}_t$ survives}\}. 
$$
\item  $\theta^{\lambda,\BT_d}_t$ {\bf survives locally/strongly} (\cite{morrow1994critical,durrett1995intermediate,liggett1996multiple}) if for all $x$
$$
 P^x(\theta^{\lambda,\BT_d}_t(x)=1 \ i.o.)>0.
$$
We define critical value
$$
\lambda_2(d)=\inf\{\lambda: \ \text{$\theta^{\lambda,\BT_d}_t$ survives strongly}\}. 
$$
\end{itemize}




\section{Statement of Main Theorems}
\label{section_2.1}
\subsection{Results for the threshold-one contact process}
\label{subsection3.1}
 Denote $\xi^{\lambda,\BT_d}_t$ as the threshold-one contact process on $\BT_d$. Observe that the system is increasing with respect to $\lambda$, additive, and attractive with respect to the natural order of configurations. And that $\delta_{\textbf{0}}$ is the only absorbing state/trivial invariant distribution. Recalling the definitions of the different critical values, it is easy to see that $\lambda_2(d)\ge \max\{\lambda_1(d),\lambda_e(d)\}$. In Section \ref{section_2} one will see that $\tilde \xi^{\lambda,\BT_d}_t$ is the time reversal of $\xi^{\lambda,\BT_d}_t$ and thus $\tilde \xi^{\lambda,\BT_d}_t$ survives from finite population if and only if $\xi^{\lambda,\BT_d}_t$ survives. Unlike the ordinary contact process which is self dual, it is clear that $\tilde \xi^{\lambda,\BT_d}_t$ is {\bf not} identically distributed as $\xi^{\lambda,\BT_d}_t$. So we first prove:

\begin{theorem}
\label{theorem duality}
For any $\lambda>0$, $\xi^{\lambda,\BT_d}_t$ survives if and only if it survives from finite population. This also implies $\lambda_e(d)=\lambda_1(d)$. 
\end{theorem}

With Theorem \ref{theorem duality}, we will no longer discriminate between survival and survival from finite population, and from now on denote such condition by {\bf survival}. The following two results give sufficient conditions for $\xi^{\lambda,\BT_d}_t$ to survive: 

\begin{theorem}
\label{survive d}
For any $d\ge 2$, $\lambda_1(d)<1/(d-1)$. 
\end{theorem}

\begin{remark}
Recall that $\lambda_1$ is the critical value for the existence of non-trivial stationary distribution. \cite[Corollary 3.2 \& Theorem 4.1]{xue2015asymptotic} proves that $\lambda_1(d)\in [1/(d+1), 1/(d-1)]$. With Theorem \ref{theorem duality}, one can see Theorem \ref{survive d} slightly improves their result by showing survival on the common upper bound. 
\end{remark}

Note that the condition of survival in Theorem \ref{survive d} is not necessarily sharp for each given $d$, although it has been shown in \cite{xue2015asymptotic} to be asymptotically sharp as $d\to\infty$. In fact, when $d=2$, $\xi^{\lambda,\BT_d}_t$ survives for substantially smaller $\lambda$: 
\begin{theorem}
\label{survive 2}
 $\lambda_1(2)<0.637$.  
\end{theorem}

The following theorem gives an almost complete picture on the complete convergence theorem of threshold-one contact process on $\BT_d$. In words, complete convergence theorem holds if and only if the $\xi^{\lambda,\BT_d}_t$ dies out or survives strongly. When the process survives weakly, we can prove the convergence theorem for any translation invariant $\mu$ or Dirac measure concentrating on the dense configuration. 
Recall the definition of dense configuration in \cite{griffeath1978limit} such that $\xi$ is dense if it satisfies 
\bea\label{d}
\exists\text{ $N>0$, for any $x\in\BT_d$, there exists $ y\in B_x(N)$ such that $\xi(y)=1$}. 
\eea where $B_x(N)$ is the ball of radius $N$ centered at $x$.
Note that the Dirac measure $\delta_\xi$ when $\xi$ is dense doesn't have to be translation invariant, and under the translation invariant measure, the configuration $\xi$ doesn't have to be dense with probability one, so that we need to give these two cases separately. \begin{theorem}
\label{complete convergence}
\begin{enumerate}
\item If $\lambda>\lambda_2$, for any initial configuration $\xi$,
\bea\label{ccc}
\lim_{t\to\infty}\delta_{\xi} S^\lambda(t)= [1-\alpha_\xi (\lambda)]\bar \nu^\lambda+\alpha_\xi(\lambda) \delta_{\textbf{0}},
\eea
where $ \alpha_\xi(\lambda)=P^{\xi}(  {\xi}^{\lambda,\BT_d}_t=  \textbf{0}\text{ for some }t)$, $S^\lambda(t)$ is the Markov semigroup and $\bar\nu^\lambda $ is the maximal invariant measure of ${\xi}^{\lambda,\BT_d}_t$. 

\item If $\lambda_1<\lambda\le\lambda_2$, for any $\mu$ satisfying $\mu$ is  translation invariant or $\mu=\delta_\xi$ where $\xi$ is dense, we have
$$
\lim_{t\to\infty}{\mu} S^\lambda(t)=[1- \alpha_\mu (\lambda)]\bar \nu^\lambda+\alpha_\mu(\lambda) \delta_{\textbf{0}},
$$where $ \alpha_\mu(\lambda)=P^\mu(\xi=\textbf{0})$.
\item If $\lambda\le\lambda_1$,
for any initial configuration $\xi$,
$$
\lim_{t\to\infty}\delta_{\xi} S^\lambda(t)=\delta_{\textbf{0}}.
$$

\end{enumerate}
\end{theorem}
{\begin{remark}\label{exist}
When $\lambda_2(d)$ does not exist, the theorem makes no sense. In fact, similar to the ordinary contact process on $\BT_d$, we can prove that $\lambda_2(d)\le\frac{1}{\sqrt{d}-1}$.
\end{remark}}
\subsection{Results for the threshold-one voter model with spontaneous death}
 Denote $\eta_t^{\varepsilon,\BT_d}$ as the threshold-one voter model with positive spontaneous death rate $\varepsilon$ on $\BT_d$, where when $\varepsilon=0$, it is the threshold-one voter model. Observe that $\eta_t^{\varepsilon,\BT_d}$ is monotone with respect to $\varepsilon$ and is attractive with respect to the natural order of configurations, so that with the definition of the three types of survivals, we can define three critical values with respect to $\varepsilon$.
\bea
&\varepsilon_1(d)=\sup\{\varepsilon: \text{$\eta^{\varepsilon,\BT_d}_t$ survives }\}. 
\\&\varepsilon_e(d)=\sup\{\varepsilon: \text{$\eta^{\varepsilon,\BT_d}_t$ survives from finite population}\}. 
\\&\varepsilon_2(d)=\sup\{\varepsilon: \text{$\eta^{\varepsilon,\BT_d}_t$ survives strongly}\}. 
\eea
Like the threshold-one contact process, we have $\varepsilon_2(d)\le \max\{\varepsilon_1(d),\varepsilon_e(d)\}$. Because the dual process is not attractive, there is no clear indication to our knowledge whether $\varepsilon_1(d)=\varepsilon_e(d)$. Since $\eta^{\varepsilon,\BT_d}_t$ can dominate the threshold-one contact process $\xi^{1,\BT_d}_t$, by Theorem \ref{survive d} and Theorem \ref{survive 2}, we can get that for any $d\ge 2$, $\varepsilon_e(d)>0$ and $\varepsilon_1(d)>0$, and by Remark \ref{exist}, one can prove that $\varepsilon_2(d)>0$ when $d\ge 5$. Next, for the the threshold-one voter model with spontaneous death, we also have the corresponding complete convergence theorem.

\begin{theorem}\label{thm62}

\begin{enumerate}
\item If $\varepsilon<\varepsilon_2$, for any initial configuration $\eta$,
$$
\lim_{t\to\infty}\delta_{\eta} S^\varepsilon(t)=[1- \alpha_\eta (\varepsilon)]\bar\nu^\varepsilon+\alpha_\eta(\varepsilon)\delta_{\textbf{0}},
$$
where $ \alpha_\eta(\varepsilon)=P^{\eta}(  {\eta}^{\varepsilon,\BT_d}_t=  \textbf{0}\text{ for some }t)$,  $S^\varepsilon(t)$ is the Markov semigroup and $\bar\nu^\varepsilon $ is the upper invariant measure of ${\eta}^{\varepsilon,\BT_d}_t$. 

\item If $\varepsilon_1\ge\varepsilon\ge\varepsilon_2$, for any $\mu$ satisfying $\mu$ is  translation invariant or $\mu=\delta_\eta$ where $\eta$ is dense, we have
$$
\lim_{t\to\infty}{\mu} S^\varepsilon(t)=[1- \alpha_\mu ]\bar \nu^\varepsilon+\alpha_\mu  \delta_{\textbf{0}},
$$where $ \alpha_\mu =\mu(\eta=\textbf{0})$.
\item If $\varepsilon>\varepsilon_1$, for any initial configuration $\eta$, we have
$$
\lim_{t\to\infty}\delta_{\eta} S^\varepsilon(t)= \delta_{\textbf{0}}.
$$
\end{enumerate}

\end{theorem}
Next we will consider the threshold-one voter model on $\BT_d$, we denote the process by $\zeta^{\BT_d}_{t}$. Since $\zeta^{\BT_d}_{t}$ can dominate the threshold-one contact process with parameter 1, by Theorem \ref{survive d} and \ref{survive 2}, and Proposition 2.11 in \cite{liggett1999stochastic}, we have 
\begin{theorem}\label{thm71}
For any $d\ge 2$, $\zeta^{\BT_d}_{t}$ coexists, i.e., there exists a nontrivial invariant measure.
\end{theorem}
{Then by Remark \ref{exist} and the domination of $\xi^{\lambda,\BT_d}_{t}$, we can get that $\zeta^{\BT_d}_{t}$ survives strongly when $d\ge 5$}.

 Define 
$$
\beta(d)=\lim\limits_{n\to\infty}[P^o(x_n\in \zeta_t^{\BT_d}\text{ for some }t)]^{1/n}.
$$
Then $\beta(d)$ exists by Theorem B22 in \cite{liggett1999stochastic}.
Like the definition in \eqref{d}, we define the configuration to be doubly dense if it satisfies
\bea\label{dense}
\exists\text{ $N>0$, for any $x\in\BT_d$, there exists $ y,z\in B_x(N)$ such that $\zeta(y)=1,\zeta(z)=0$}. 
\eea
\begin{theorem}
\begin{enumerate}\label{thm72}
\item When $\zeta_t^{\BT_d}$ survives strongly, for any initial configuration $\zeta$,
 $$\lim_{t\to\infty}\delta_{\zeta} U(t)= \alpha_\zeta\delta_0+\beta_\zeta\delta_1+(1-\alpha_\zeta-\beta_\zeta)\mu^d_{1/2} $$ 
  where $\alpha_\zeta=P^\zeta( {\zeta}^{\BT_d}_t=  \textbf{0}\text{ for some }t)$, $\beta_\zeta=P^\zeta( {\zeta}^{\BT_d}_t=  \textbf{1}\text{ for some }t)$, $U(t)$ is the Markov semigroup of $\zeta_t^{\BT_d}$, and $\mu^d_{1/2} $ is the limiting distribution of the process starting from the product distribution with parameter 1/2.
\item When $\beta(d)< \frac{1}{\sqrt{d}}$, for any $\mu$ satisfying $\mu$ is translation invariant or $\mu=\delta_\zeta$ where $\zeta$ is doubly dense,
 $$
\lim_{t\to\infty}\mu U(t)= \alpha_{\mu}\delta_{\textbf{0}}+\beta_{\mu}\delta_{\textbf{1}}+(1- \alpha_{\mu}- \beta_{\mu})\mu^d_{1/2},
$$ where $\alpha_\mu=\mu(\zeta=\textbf{0})$, $\beta_\mu=\mu(\zeta=\textbf{1})$.
 \end{enumerate}
 \end{theorem}

\section{Preliminaries and Graphical Representations}
\label{section_2}

\subsection{Preliminaries}

To precisely define the particle system, let $\BT_d$ be the homogenous tree where each vertex has $d+1$ neighbors. One can order $\BT_d$ such that for each vertex $x$, it has exactly one {\bf parent}, denoted as $\overleftarrow{x}$ and $d$ {\bf children}. At the same time, one can also label all the $d$ children of each vertex by 1 to $d$ according to this order. If a vertex $x$ can be traced back to another vertex $y$ through a path going up in the ``family tree", we call $x$ the {\bf ancestor} and $y$ the {\bf descendant}. For vertices $x,y\in \BT_d$, we write $x\sim y$ if $y$ is a neighbor of $x$, and we use $|x-y|$ to denote the distance between $x$ and $y$ on $\BT_d$, especially, when $x=o$, we denote the distance between $y$ and $o$ as $|y|$. 

According to Part 1, Section 4 of \cite{liggett1999stochastic}, one may embed $\BZ$ into $\BT_d$ by assigning consecutive integers to a doubly infinite connected lineage within $\BT_d$ such that for vertex $x$ and $y$ which are parent and child, we always have $l(y)=l(x)+1$, where $l(x)$ is the integer number assigned to vertex $x$. Thus, one can also see $l(x)$ as the generation number of vertex $x$ along this lineage, and denote by $o$ the vertex with $l(x)=0$. And for each $k\in \BZ$, denote by $x_k$ the vertex on this lineage with $l(x_k)=k$. With generation numbers defined along one specific lineage, one can now discriminate all other vertices with the following finite sequence of integers: for each point $z\in \BT_d$, define $n_z\in \BZ^+\cup \{0\}$ such that $x_{-n_z}$ is the most recent common ancestor of $o$ and $z$. In particular, $z$ is an descendent of $o$ if and only if $n_z=0$. If $z\not= x_{-n_z}$, note that there is a unique simple path $\vec y$ from $x_{-n_z}$ to $z$ denoted by $x_{-n_z}=y(0), y(1),\cdots, y(m_z)=z$. Now define $\big(r(1),r(2),\cdots, r(m_z)\big)\in \{1,2,\cdots, d\}^{m_z}$ such that for each $i$, $y(i)$ is the $r(i)-$th descendent of $y(i-1)$. From the definitions above, one can see that each vertex $z\in \BT_d$ now has a unique correspondence with $\big(n_z, r(1),r(2),\cdots, r(m_z)\big)$, while $|z|=n_z+m_z$. One may note that the discussion above also gives an explicit construction of the infinite homogeneous tree $\BT_d$.  

For any $A\subseteq \BT_d$, denote by $\xi_A$ the configuration in $\{0,1\}^{T_d}$ such that $\xi_A(x)=1$ if and only if $x\in A$. 
\subsection{Coalescing duality}
The graphical representation for the threshold-one contact process can be referred to Section 6 of Chapter I in \cite{liggett1985}. Since the threshold-one contact process is additive, it has coalescing duality. Now we construct the threshold-one contact process and its dual process as follows.
 
 For any $x\in \BT_d$, let $\{T_n^{x,\CN_x}\}$ and $\{T_n^{x,\emptyset}\}$ be two independent Poisson flows with rate $\lambda$ and 1 respectively. Let $\CN_x=\{y:|y- x|\le 1\}$ be the neighborhood of $x$. At each event time $s$ of $T_n^{x,\CN_x}$ we draw an arrow from $y$ to $x$ for every $y\in \CN_x\backslash\{x\}$, which can be thought that the particle at $y$ gives a birth at $x$ at time $s$. And at each event time s of $T_n^{x,\emptyset}$ put a $\delta$ at $(x,s)$ where $\delta$ represents deaths. We say that there is an open path from $(x,0)$ to $(y,t)$ if there is a pair of sequences $x_0=x,\dots,x_n=y$ and $s_0=0<s_1<\dots s_n<s_{n+1}=t$ such that:
  \begin{itemize} 
 \item For $1\leq m\leq n$ there is an arrow from $x_{m-1}$ to $x_{m}$ at time $s_m$.
 \item For $0\leq m\leq n$ there is no $\delta's$ in $\{x_m\}\times(s_{m},s_{m+1})$.
 \end{itemize}
 
To get the threshold-one contact process, for any $A\subset \BT_d$, let
$$
\xi_t^{\lambda,A,\BT_d}(x)=\ind_{\{\text{ there is an open path from $(z,0) $ to $(x,t)$ for some $z\in A$}\}}.
$$
 To get the dual process up to time $t$, reverse the directions of the arrows and time in the graph representation before time $t$, for any $s\leq t$, for any $B\subset \BT_d$, let
$$
\tilde{\xi}^{\lambda,B,\BT_d}_{t,s}(x)=\ind_{\{\text{ there is an open path from $(z,t) $ to $(x,t-s)$ for some $z\in B$ }\}}
$$
Since for any $s\leq t_1\leq t_2$, $\tilde{\xi}^{\lambda,B,\BT_d}_{t_1,s}$ and $\tilde{\xi}^{\lambda,B,\BT_d}_{t_2,s}$ have the same distribution, we can define the dual process $\tilde{\xi}^{\lambda,B,\BT_d}_{t}\triangleq\tilde{\xi}^{\lambda,B,\BT_d}_{t,0}$. Note that $\xi_t^{\lambda,A,\BT_d}$ has the infinitesimal generator $\tilde\Omega^\lambda$ as follows: for any local function $f$,
\begin{equation}
\label{generator_dual} 
\Omega^\lambda f(\xi)=\sum_{x\in \BT_d} \Big(\ind_{\{\xi(x)=1\}}[f(\xi^x)-f(\xi)]+ \lambda\ind_{\{\xi(x)=0,\sum\limits_{y\sim x}\xi(y)\ge 1\}} [f(\xi^x)-f(\xi)]\Big) 
\end{equation}
where $\xi^x$ has the same state as $\xi$ at any vertex except $x$.
And $\tilde{\xi}^{\lambda,B,\BT_d}_{t}(x)$ has the infinitesimal generator $\tilde\Omega^\lambda$ as follows: for any local function $f$,
\begin{equation}
\label{generator_dual} 
\tilde\Omega^\lambda f(\xi)=\sum_{x\in \BT_d} \ind_{\{\xi(x)=1\}}\left([f(\xi^ x)-f(\xi)]+ \lambda [f(\xi\cup \xi_{\CN_x})-f(\xi)]\right). 
\end{equation}
where $\xi\cup\eta=\xi\vee\eta$.

Let $P^{\xi_A}(\cdot)$ and $\tilde P^{\xi_B}(\cdot)$ denote the distribution of ${\xi}^{\lambda,A,\BT_d}$ and $\tilde{\xi}^{\lambda,B,\BT_d}$ respectively, $E^{\xi_A}(\cdot)$ and $\tilde E^{\xi_B}(\cdot)$ denote the corresponding expectations. When $A=\{x\}$ is a singleton, we abbreviate $P^{\xi_{\{x\}}}$, $E^{\xi_{\{x\}}}$ and ${\xi}^{\lambda,A,\BT_d}$ to $P^x$, $E^x$ and ${\xi}^{\lambda,x,\BT_d}$ for convenience, and we sometimes omit $A$ in ${\xi}^{\lambda,A,\BT_d}$. From the construction, we can get the following important dual relation
 \bea\label{20}
 P^{\xi_A} ( \xi^{\lambda,\BT_d}_t\cap\xi_B\ne \textbf{0})=\tilde P^{\xi_B} ( \tilde{\xi}^{\lambda,\BT_d}_t\cap \xi_A\ne\textbf{0} ) \quad\forall A, B \subset \BT_d
 \eea
where $\xi\cap\eta=\xi\land\eta$.

The graph representation is a useful tool to get the monotonicity of $\xi_t^{\lambda,\BT_d}$ in $\lambda$ and the initial distribution such that
 \begin{itemize} 
  \item For any $\lambda_{(1)}\leq\lambda_{(2)}$, regarding the Poisson arrows with parameter $\lambda_{(2)}$ as superpositions of independent Poisson arrows with parameter $\lambda_{(1)}$ and $\lambda_{(2)}-\lambda_{(1)}$, we can couple $\xi_t^{\lambda_{(1)},\BT_d}$ and $\xi_t^{\lambda_{(2)},\BT_d}$ on the common graph representation, so that 
 \bea \label{23}
 \xi_t^{\lambda_{(1)},\BT_d}\subseteq \xi_t^{\lambda_{(2)},\BT_d}
  \eea
  for any $t$ when both starting from $A$. 
\item For any finite subsets $A\subseteq B$, we can couple ${\xi}^{\lambda,A,\BT_d}$ and ${\xi}^{\lambda,B,\BT_d}$ with the same $\lambda$ such that 
\bea\label{25}
\xi_t^{\lambda,A,\BT_d}\subseteq \xi_t^{\lambda,B,\BT_d}
\eea
for any $t$ by using the same Poisson flows on a common graph representation.   
\end{itemize}
The same results also hold for the dual process. Moreover, $\xi^{\lambda,\BT_d}_t$ and $\tilde \xi^{\lambda,\BT_d}_t$ are additive in the sense that 
$$
\xi^{\lambda,A,\BT_d}_t=\cup_{x\in A}\xi^{\lambda,x,\BT_d}_t, \ \tilde \xi^{\lambda,A,\BT_d}_t=\cup_{x\in A}\tilde \xi^{\lambda,x,\BT_d}_t.
$$
Let $A=\BT_d$ in \eqref{20}, we obtain
\begin{equation}
\label{p} 
\tilde P^{\xi_B}(\tilde{\xi}^{\lambda,\BT_d}_t\ne\textbf{0})=\delta_\mathbf{1}S^\lambda(t)(\xi:\xi\cap \xi_B\ne\textbf{0}).
\end{equation}
Thus the limit of $\delta_\mathbf{1}S^\lambda(t)$ exists and denote it as $\bar{\nu}^\lambda$ such that 
\bea\label{21}
\tilde P^{\xi_B}(\tilde{\xi}^{\lambda,\BT_d}_t\ne\textbf{0}\text{ for all } t)=\bar{\nu}^\lambda(\xi:\xi\cap \xi_B\ne\textbf{0})
\eea
for any $B$.

\subsection{Annihilating duality}
Since the threshold-one voter model w/o death is not additive, it has no coalescing duality. However, it has another important duality, the annihilating duality (see Chapter 3.4 of \cite{liggett1985}), which was introduced in \cite{cox1991nonlinear}. 

 For any $x\in \BT_d$, let $\{T_n^{x,S_x}\}_{S_x\subseteq \CN_x}$ and $\{T_n^{x,\emptyset}\}$ be independent Poisson processes with rate $\beta(S_x)$ and $\varepsilon$ respectively, where \begin{center}

  $\beta(S_x)= \left \{
    \begin{array}{ll}
   \frac{1}{2^{d}}  \quad &S_x\subset\CN_x \text{ and $|S_x|\ge 2$ is even}，\\
   0\quad &otherwise   \end{array}\right.$.\\
\end{center}

 At any event time $t\in\{T_n^{x,S_x}\}$ draw an arrow from $y$ to $x$ for every $y\in S_x$ and if $x\in S_x$ we put a $\delta$ at $(x,t)$, while if $t\in\{T_n^{x,\emptyset}\}$, put a $\delta$ at $(x,t)$.
  
  Let\begin{center}$\eta_t^{\varepsilon,A,\BT_d}(x)=1_{\{\text{ there are odd number of paths from $(z,0)$ to $(x,t)$ for some $z\in A$}\}}$.\end{center} 
Reverse the direction of the arrows, let
\bea\label{2}
\tilde{\eta}_{t,s}^{\varepsilon,B,\BT_d}(y)=1_{\{\text{ there are odd number of paths from $(z,t)$ to $(y,s)$ for some $z\in B$}\}}.
\eea

Then the dual process starting from $\eta_B$ can be defined as $\tilde{\eta}_{t,0}^{\varepsilon,B,\BT_d}$. From the above argument we can define the infinitesimal generator of $\eta_t^{\varepsilon,A,\BT_d}$ 
as: for any local function $f$,
\begin{equation}
\label{generator_dual} 
\Omega^\varepsilon f(\eta)=\sum_{x\in \BT_d}\Big(\sum_{S_x} \beta(S_x)[f(\eta\triangle \eta_{S_x})-f(\eta)]+\varepsilon\ind_{\{\eta(x)=1\}}[f(\eta^x)-f(\eta)]\Big)
\end{equation} where $\eta\triangle \eta_{S_x}(y)=1$ iff $\eta(y)[1-\eta_{S_x}(y)]=1$ or $[1-\eta(y)]\eta_{S_x}(y)=1$. $\eta$ is odd means that there are odd number of 1's in $\eta$.

And the infinitesimal generator of $\tilde\eta_t^{\varepsilon,A,\BT_d}$ 
can be defined as: for any local function $f$,
\begin{equation}
\label{generator_dual} 
\tilde\Omega^\varepsilon f(\eta)=\sum_{x\in \BT_d}\Big(\sum_{S_x} \beta(S_x)[f(\eta\triangle \eta_{S_x})-f(\eta)]+\varepsilon\ind_{\{\eta(x)=1\}}[f(\eta^x)-f(\eta)]\Big).\end{equation}
 The threshold-one voter model is the special case when $\varepsilon=0$.

From the construction above, ${\eta}^{\varepsilon,\BT_d}_t$ and $\tilde{\eta}^{\varepsilon,\BT_d}_t$ can be seen as annihilating branching processes in the sense that:
when two particles meet, they annihilate.

Thus we get the duality relation:
\bea\label{2_cd}P^{\eta_A}( \eta_t^{\varepsilon,\BT_d}\cap \eta_B\text{ is odd} )=\tilde P^{\eta_B}( \tilde{\eta}_t^{\varepsilon,\BT_d}\cap \eta_A\text{ is odd} ) \quad\forall A, B.\eea

\section{The threshold-one contact process}
\label{section_3}

The proofs in this section are mostly adaptions of the corresponding results and methods developed for the ordinary contact process on $\BT_d$, especially those in \cite{liggett1999stochastic} and \cite{liggett1996multiple}.

\subsection{Equivalence of survivals}
In this subsection, we extend the arguments in Section 4 of Part I in \cite{liggett1999stochastic} to $\xi^{\lambda,\BT_d}_t$ and $\tilde \xi^{\lambda,\BT_d}_t$ with $\rho=1$ and prove Theorem \ref{theorem duality}. The proofs of Lemma \ref{le31} and \ref{le33} are parallel with no significant difference with those in \cite{liggett1999stochastic} since all the three processes $\xi^{\lambda,\BT_d}_t$, $\tilde \xi^{\lambda,\BT_d}_t$ and the ordinary contact process are attractive and additive spin systems, so we omit the proofs here.

\begin{lemma}[Adaption of Proposition 4.27 (b), \cite{liggett1999stochastic}]
\label{le31}
\bea\label{phi}
\Phi(\lambda)=\lim\limits_{t\to\infty}E^o|\xi^{\lambda,\BT_d}_t|^{1/t}
\eea exists and there exists
$ C<\infty$ such that 
\bea\label{33}
[\Phi(\lambda)]^t\leq E^o|\xi^{\lambda,\BT_d}_t|\leq C[\Phi(\lambda)]^t.
\eea
\end{lemma}

\begin{lemma}[Adaption of Proposition 4.39, \cite{liggett1999stochastic}]
\label{le33}
\bea&\lambda_e(d)=\sup\{\lambda:\Phi(\lambda)\leq1\},
\\&\tilde{\lambda}_e(d)=\sup\{\lambda:\tilde{\Phi}(\lambda)\leq1\}\label{6}.\eea
Moreover, $\xi^{\lambda,\BT_d}_t$ and $\tilde\xi^{\lambda,\BT_d}_t$ die out at the first critical value.
\end{lemma}

{\bf Proof of Theorem \ref{theorem duality}}: By duality, it suffices to prove $\lambda_e(d)=\tilde{\lambda}_e(d)$. When starting from a single vertex $\{o\}$, it is easy to calculate that 
\bea\label{xx}
E^o|\xi^{\lambda,\BT_d}_t|=\sum\limits_{n=0}^{\infty}(d+1)d^nP^o(x_n\in \xi^{\lambda,\BT_d}_t)=\sum\limits_{n=0}^{\infty}(d+1)d^nP^o(x_n\in \tilde{\xi}^{\lambda,\BT_d}_t)=E^o|\tilde{\xi}^{\lambda,\BT_d}_t|.
\eea
By the definition of $\Phi$ in \eqref{phi}, $\Phi(\lambda)=\tilde{\Phi}(\lambda)$. Thus it follows from \eqref{6} that
\bea\label{366}
\lambda_e(d)=\tilde{\lambda}_e(d).
\eea
Moreover, by Lemma \ref{le33}, both $\xi^{\lambda,\BT_d}_t$ and $\tilde{\xi}^{\lambda,\BT_d}_t$ die out at the first critical value so that $\xi^{\lambda,\BT_d}_t$ neither survives nor survives from finite population at $\lambda_e(d)$ by duality.\qed

\begin{remark}
Define another two critical values
\begin{align*}&\lambda_T=\sup\{\lambda:\left(\int_0^\infty\e^o|\xi^{\lambda,\BT_d}_t|dt\right)<\infty\}\\&
\tilde\lambda_T=\sup\{\lambda:\left(\int_0^\infty\e^o|\tilde\xi^{\lambda,\BT_d}_t|dt\right)<\infty\}
\end{align*}
Then by \eqref{xx}, we get $\lambda_T=\tilde\lambda_T$. Once we have
\bea\label{3_1}
\lambda_e=\lambda_T, \tilde\lambda_e=\tilde\lambda_T,
\eea Theorem \ref{theorem duality} holds immediately.
For any transitive graph with bounded degree, Aizenman and Jung in \cite{aizenman2007critical} showed that for the ordinary contact process,
$$\hat\lambda_e=\hat\lambda_T$$where $\hat\lambda_e,\hat\lambda_T $ represent the corresponding critical values for the ordinary contact process. Applying their proof to the threshold-one contact process and its dual process, we should also be able to obtain \eqref{3_1}. Thus Theorem \ref{theorem duality} holds for any transitive graph with bounded degree, such as $\BZ^d$.
\end{remark}

\subsection{Conditions for survival}\label{sec_32}
In this subsection, we will give conditions for the dual process $\tilde\xi^{\lambda_d,\BT_d}_t$ to survive, as a result, we can get an upper bound for the critical value $\lambda_1(d)$. 

{\bf Proof of Theorem \ref{survive d}:}
\begin{proof} 
For any non-empty $A\in Y$, define $g(A)=|A|$. Then for $\lambda_d=(d-1)^{-1}$ and our process $\xi^{\lambda_d,\BT_d}_t$ on $\BT_d$, we can similarly define
\bea
\label{infinitesimal cardinality}
\begin{aligned}
h^{\lambda_d}(A)&=\frac{d}{dt}E^{\xi_A} \left[g\left(\xi^{\lambda_d,\BT_d}_t\right)\right]\Big|_{t=0}\\
&=\lambda_d |\partial A|-|A|\\
&\ge \frac{1}{d-1}[(d-1)|A|+2]-|A|
\\&=\frac{2}{d-1}. 
\end{aligned}
\eea
The third inequality comes from the fact that $|\partial A|\ge(d-1)|A|+2$, which can be proved by induction. First when $|A|=1$, $|\partial A|=d+1=(d-1)|A|+2$. Then we assume that the fact is true for all $|A|\le n-1$. We come to the case when $|A|=n$. Assume $x\in A$ has the maximum generation,  i.e.  
$$
m_{x}-n_{x}=\max\{m_y-n_y: \ y\in A\}. 
$$
If there is more than one such point, we choose the ``smallest" one according to the lexicographical order of $(n_z, r(1),r(2),\cdots, r(m_z))$. It follows that 
$$
  |\partial A|\ge  |\partial (A\cup\{x\})|-1+d
  $$
  where the inequality holds iff $\overleftarrow x \in A$. Hence by the inductive hypothesis we have
  \begin{align*}
   |\partial A|&\ge  |\partial (A\cup\{x\})|-1+d
   \\&\ge(d-1)(n-1)+2+d-1
   \\&=n(d-1)+2
   \\&=(d-1)|A|+2.
   \end{align*}
  Now the fact is true for all $A$ satisfying $|A|=n$ and thus for all finite subset $A$.
From the transition rates above, one can see that the stochastic process $|\xi^{\lambda_d,\BT_d}_t|$ dominates a Markov process $\hat A^d_t$ on $\BZ^+\cup\{0\}$ with 0 as absorbing state and transition rates as follows: for each $n\in \BZ^+$, 
$$
\left\{
\begin{aligned}
&n\rightarrow n+1 \text{ at rate }\hspace{0.1 in}n+\frac{2}{d-1}\\
&n\rightarrow n-1 \text{ at rate }\hspace{0.1 in}n\\
\end{aligned}
\right.
$$

Note that in Lemma \ref{le33} we have proved that survival of $A^{\lambda_{d},\BT_d}_t$ is equivalent to $\Phi(\lambda)>1$. Define $P^n$ as the expectation of random walk starting from $n$. By Lemma \ref{le31} and \ref{le33} it suffices to prove that $\lim\limits_{t\to\infty}E^o\left|A^{\lambda_d,\BT_d}_t\right|= \infty$, which, by the discussion above, is immediate if one can show that 
\bea
\label{expectation_limit}
\lim_{t\to\infty}E^1\left|\hat A^{d}_t\right|= \infty. 
\eea
Again, note that for $\hat A^d_t$, its infinitesimal mean is given by 
$$
\mu^d(n)=\frac{d}{dt} E_n \hat A^d_t\Big|_{t=0}=\frac{2}{d-1}\ind_{\{n\not=0\}}. 
$$
Recalling that 0 is absorbing for $\hat A^d_t$, so for stopping time 
\bea\label{hit}\hat \tau_0=\inf\{t: \hat A^d_t=0\}\eea and any $t\in (0,\infty)$, by Dynkin's formula,
$$
\begin{aligned}
E^1 \hat A^d_t= 1+E^1\left[\int_0^t \frac{2}{d-1}\ind_{\{\hat \tau_0>s\}}ds \right]=1+ \frac{2}{d-1}\int_0^t P^1(\hat \tau_0>s)ds.
\end{aligned}
$$
Note that $\int_0^\infty P^1(\hat \tau_0>s)ds=E^1[\hat \tau_0]$, thus in order to prove \eqref{expectation_limit}, it is sufficient to prove that $E^1[\hat \tau_0]=\infty$. 

Moreover, note that $\hat A^d_t$ further dominates $S_t$ which is a (inhomogeneously) rescaled continuous time simple random walk with 0 as absorbing state and transition rates as follows: for each $n\ge 1$, 
$$
\frac{d}{dt} P^n(S_t=n+1)\Big|_{t=0}=\frac{d}{dt} P^n(S_t=n-1)\Big|_{t=0}=n. 
$$
Let $\tau_0=\inf\{t: S_t=0\}$. Thus we only need to prove that $E^1[\tau_0]=\infty$. Now define the following sequence of stopping time for $S_t$:
$$
\tau_n=\inf\{t\ge 0: \ S_t=n\}, \ \ \forall n\ge 1
$$
with the convention $\inf\emptyset=\infty$, and disjoint events $A_n=\{\tau_n<\infty, \ \tau_{n+1}=\infty\}$. Then 
\bea
\label{0_recurrent_1}
E^1[\tau_0]=\sum_{n=1}^\infty E^1[\tau_0\ind_{A_n}]. 
\eea
Now for each sufficiently large $n$, note that on the event $A_n$, $\tau_n<\tau_0<\tau_{n+1}=\infty$, which implies that 
\bea
\label{0_recurrent_1.5}
E^1[\tau_0\ind_{A_n}]\ge  E^1[\tau_n\ind_{A_n}].
\eea
And by strong Markov property, 
\bea
\label{0_recurrent_1.6}
E^1[\tau_n\ind_{A_n}]=E^1[\tau_n\ind_{\{\tau_n<\infty\}}]P^n(\tau_0<\tau_{n+1})=\frac{1}{n+1} E^1[\tau_n\ind_{\{\tau_n<\infty\}}],
\eea
where the last equality is a result that the embedded chain of $S_t$ is a simple random walk with 0 as absorbing state. 

Now define $Y_n=\# $of transitions within $[0,\tau_n\wedge \tau_0]$, and let $X_n, n\ge 0$ be a discrete simple random walk on $\BZ$ with stopping time 
$$
T_n=\min\{k: \ X_k=n\}, \ \forall n\ge 0. 
$$
Then again by the fact that the embedded chain of $S_t$ is a simple random walk with  0 as absorbing state, we have for all $m\ge 1$, 
\bea
\label{0_recurrent_2}
P^1(\tau_n<\infty, Y_n\ge m)=P^1(T_n<T_0, T_n\ge m).
\eea
Note that $P^1(T_n<T_0)=n^{-1}$, while 
$$
P^1(T_n<m)= P^0(T_{n-1}<m)\le P^0\left(\max_{i\le m}X_i\ge n-1 \right). 
$$
By reflection principle, 
$$
P^0\left(\max_{i\le m} X_i\ge n-1 \right)=2 P^0(X_m\ge n-1). 
$$
Now let $m_n=[n^2/(3\log n)]$, by Chernoff's inequality, 
\bea
\label{0_recurrent_3}
 P^0(X_{m_n}\ge n-1)\le \exp\left(-\frac{(n-1)^2}{m_n}\right)\le \exp(-2\log n)=\frac{1}{n^2}, 
\eea
which implies that 
\bea\label{4}
P^1(T_n<T_0, T_n\ge m_n)\ge P^1(T_n<T_0)-P^1(T_n<m_n)\ge \frac{1}{n}-\frac{2}{n^2}\ge \frac{1}{2n}
\eea
when $n$ is large enough.
At the same time, note that given any nonzero and nearest neighbor trajectory $1=x_{(0)}, x_{(1)}, \cdots, x_{(m)}=n$ with $x_{(0)}, x_{(1)}, \cdots, x_{(m-1)}\in [1,n-1]$ as the embedded chain of $S_t$ within $[0,\tau_n]$, the conditioned expectation of $\tau_n$ is given by 
$$
\frac{1}{x_{(0)}}+\frac{1}{x_{(1)}}+\cdots+\frac{1}{x_{(m-1)}}\ge \frac{m}{n}. 
$$
Moreover, in the event $\{\tau_n<\infty, Y_n\ge m_n\}$, each trajectory of the embedded chain within $[0,\tau_n]$ has to satisfy the condition above with $m=m_n$. Thus 
\bea
\label{0_recurrent_4}
E\left[\tau_n\big| \tau_n<\infty, Y_n\ge m_n\right]\ge \frac{m_n}{n}\ge \frac{n}{4\log n}.
\eea
By \eqref{0_recurrent_1.5}, \eqref{0_recurrent_1.6}, \eqref{0_recurrent_2} and \eqref{0_recurrent_4} we have
\bea
\label{0_recurrent_5}
\begin{aligned}
E^1[\tau_0\ind_{A_n}]\ge \frac{1}{n+1} E^1[\tau_n\ind_{\{\tau_n<\infty\}}]&\ge \frac{1}{n+1} P^1(\tau_n<\infty, Y_n\ge m_n)  E\left[\tau_n\big| \tau_n<\infty, Y_n\ge m_n\right]\\
&\ge  \frac{1}{8(n+1)\log n}
\end{aligned}
\eea
when $n$ is large enough, which is un-summable. Thus $E^1[\tau_0]=\infty$ and the proof of Theorem \ref{survive d} is complete.
\end{proof}

{\bf Idea of the proof of Theorem \ref{survive 2}:}
Recall in Theorem \ref{theorem duality}, we have prove that $\xi^{\lambda,\BT_d}_t$ survives if and only if $\tilde{\xi}^{\lambda,\BT_d}_t$ survives, and in Lemma \ref{le33} we prove that ${\xi}^{\lambda_1(2),\BT_d}_t$ dies out. Thus it suffices to prove that when $d=2$ and $\lambda\geq 0.637$, the dual process $\tilde{\xi}^{\lambda,\BT_d}_t$ survives. The constant $0.637$ is the same as the weaker upper bounds for survival of the ordinary contact process on $\BT_2$ found in \cite{liggett1996multiple}. The key ingredient in this proof is to show that in the finite system, although $\tilde \xi^{\lambda,\BT_d}_t$ and the ordinary contact process have different transition dynamic, their infinitesimal generators will behave exactly the same when applying on the special limiting test function $f$ defined in \cite{liggett1996multiple}. Specifically, 
$$
 h(A)=\frac{d}{dt}E^{\xi_A}f(\tilde{A}^\lambda_t)|_{t=0}
 $$
 is the same as (4.2) in \cite{liggett1996multiple}, the function $h$ corresponding to the ordinary contact process. 
 
  Let $\tau_a=\inf\{t:f(\xi_t^A)\le a\}$ where $a$ is large enough. As a result, we can show that for any finite set $A$, $f(\tilde{\xi}^{\lambda,A,\BT_d}_{t\land\tau_a})^{-1}$ is a nonnegative supermartingale whose transition rate has uniform lower bound when $\lambda\geq 0.637$. Applying the optional stopping theorem we obtain that $P^{\xi_A}(\tau_{a}=\infty)>0$ when $f(A)>a$. Since $f(\textbf{0})=0$, we get $P^{\xi_A}(\tau_{\textbf{0}}=\infty)>0$. So we have proved that $\tilde{\xi}^{\lambda,\BT_d}_t$ survives when $\lambda\geq 0.637$. Since the approach follows the approach used in Theorem 1.2 (a) in \cite{liggett1996multiple}, for completeness,  
we put the detailed proof in Appendix (A).

\subsection{Results for strong survival}
\label{section_4}
It is easy to generalize some results for the ordinary contact process on $\BT_d$, which are proved in \cite{liggett1999stochastic}, to all attractive spin systems we considered in this paper(see Remark \ref{rmk62}). For convenience, we just prove for the threshold-one contact process.

Let $\BT^d_x$ be the subtree containing $x$ together with all its descendants and $\hat{\BT}^d_x=(\BT_d\backslash\BT^d_x)\cup\{x\}$. Denote $\xi^{\lambda,\hat{\BT}^d_x}_t$ as the process restricted on $\hat\BT^d_x$ which is always 0 outside. 

Define
$$
\beta(\lambda,d)=\lim\limits_{n\to\infty}[P^o(x_n\in\xi^{\lambda,\BT_d}_t \text{ for some }t)]^{1/n},
$$ then it exists by Theorem B22 in \cite{liggett1999stochastic}.

Among the key statements proven below is
\begin{lemma}[Adaption of Proposition 4.57, \cite{liggett1999stochastic}]\label{le64} If $\beta(\lambda,d)>\frac{1}{\sqrt{d}}$, then\begin{center} $\inf\limits_t P^o(o\in \xi^{\lambda,\BT^d_o}_t)>0$.\end{center}
\end{lemma}

Since the proof follows the approach used in Proposition 4.57 in \cite{liggett1999stochastic}, for completeness, we put it in Appendix (B).

Once we have Lemma \ref{le64}, immediately we have the next  three lemmas.
\begin{lemma}
\label{corollary_equivalent}
\bea
 P^o(o\in \xi^{\lambda,\BT_d}_t\text{ i.o.})>0\text{ iff }\inf_tP^o(o\in \xi^{\lambda,\BT_o^d}_t)>0.\label{628}
 \eea
\end{lemma}

\begin{proof}
 When $\inf_tP^o(o\in \xi^{\lambda,\BT^d_o}_t)>0$, $P^o(o\in \xi^{\lambda,\BT_o^d}_t\text{ i.o.})>0$ is immediate simply by definition. To prove the other direction, when $P^o(o\in \xi^{\lambda,\BT_d}_t\text{ i.o.})>0$, by a simple argument, we have
$$
P^o(x_n\in \xi^{\lambda,\BT_d}_t \text{  i.o.})=P^o(o\in \xi^{\lambda,\BT_d}_t\text{  i.o.})
$$
for all $n\ge 1$. Recalling the definition of $\beta(\lambda,d)$, 
$$
\beta(\lambda,d)\ge \lim_{n\to\infty}[P^o(x_n\in \xi^{\lambda,\BT_d}_t \text{  i.o.})]^{1/n}=1>\frac{1}{\sqrt{d}}.
$$
Thus from Lemma \ref{le64} we have proved the other direction. 
\end{proof}

\begin{lemma}\label{lem53}
When
$$
P^o(o\in\xi^{\lambda,\BT^d}_t\text{ i.o.})>0, 
$$

$$
\lim\limits_{n\to\infty}\sup_tP^{\xi_{B(n)}}(|\xi_t^{\lambda,\BT_d}\cap B(n)|\le k)= 0$$  for any $k$ where $B(n)$ is the ball of radius $n$ centered at $o$. \end{lemma}
\begin{proof}
By Lemma \ref{corollary_equivalent}, 
$$
\inf_{t\ge 0}P^o(o\in\xi_t^{\lambda,\BT^d_o})\ge p>0.
$$

Since $\BT^d_x\cap\BT^d_y=\emptyset$ for any $ x\ne y$ but $|x|=|y|$, $\xi_t^{\lambda,\BT^d_x}$ and $\xi_t^{\lambda,\BT^d_y}$ are independent, so that for any $k$,

\begin{align*}
\sup_tP^{\xi_{B(n)}}(|\xi_t^{\lambda,\BT_d}\cap B_o(n)|\le k)
&\le \sup_tP^{\xi_{B(n)}}(|(\cup_{|x|=n}\xi_t^{\lambda,x,\BT_x^d})\cap B(n)|\le k)
\\&\le P(B((d+1)^n,p)\ge k)
\end{align*}
 where $B((d+1)^n,p)$ is the Binomial distribution with parameter $(d+1)^n$ and $p$, which converges to 0 as $n\to\infty$.

\end{proof}

\begin{lemma}\label{lem44}
 When $\xi_t^{\lambda,\BT_d}$ survives strongly,
  \begin{center}$\lim\limits_{n}\sup\limits_{|A|=n}P^{\xi_A}(\xi_t^{\lambda,\BT_d}\ne\textbf{0}\text{ for any }t)=1$.\end{center}
\end{lemma}

\begin{proof}

Assume $A\subset\BT_o^d$ and $|A|=n$. 

Define
\bea
&H(A)=\{x\in A:\text{ all children of $x$ are in $\bar{A}$}\},
\\&F(A)=A\backslash H(A),
\\&\hat A=\{x\in A,\BT_x^d\cap\bar A=\{x\}\}
\\&p_1=P^o(\xi^{\lambda,\BT^d_{x_1}\cup\{o\}}_t\text{ dies out})
\\&p_2=P^o( \xi^{\lambda,\BT^d_o}_t\text{ dies out})
\eea
where $x_1$ is a child of $o$ and $\bar A$ is the minimum connected component containing $A$. 

From Lemma \ref{corollary_equivalent}, $ \xi^{\lambda,\BT_d}_t$ survives strongly implies $ \xi^{\lambda,\BT^d_o}_t$ survives strongly so that $p_2<1$ and thus $p_1<1$.

 If $ |F(A)|\geq \frac{n}{2}$, for any $x\in F(A)$, we can choose a child $x^*\in\bar{A}^c$, thus the processes $\xi^{\lambda,x,\BT^d_{x^*}\cup\{x\}}_t,{x\in F(A)}$ are identically independent. Then by monotonicity,
\bea\label{F}
P^{\xi_A}(\xi^{\lambda,\BT_d}_t \text{ dies out})&\leq P(\forall x\in F(A),\xi^{\lambda,x,\BT^d_{x^*}\cup\{x\}}_t\text{ dies out})
\\&=\prod_{x\in F(A)}P( \xi^{\lambda,x,\BT^d_{x^*}\cup\{x\}}_t\text{ dies out})\\&\leq p_1^{\frac{n}{2}}.\eea
 Otherwise, when $|H(A)|\geq \frac{n}{2}$, we claim that $|\hat{A}|\geq |H(A)|$. We prove the claim by induction. First when $|A|=1$, we have $|H(A)|=0$, so that $|\hat{A}|\geq |H(A)|$ holds immediately. Next we assume that the claim is true for any $A$ satisfying $|A|=n$. Now we come to the case when $|A|=n+1$. Assume $x\in A$ has the maximum generation, i.e.  
$$
m_{x}-n_{x}=\max\{m_y-n_y: \ y\in A\}. 
$$
If there is more than one such point, we choose the ``smallest" one according to the lexicographical order of $(n_z, r(1),r(2),\cdots, r(m_z))$.

Then 
\begin{equation}\label{eq1}
|H(A\backslash\{x\})|=
\begin{cases}
|H(A)|-1&\text{ if } \overleftarrow x\in H(A)\\
 |H(A)|&\text{ otherwise }.
 \end{cases}
 \end{equation} 
where $\overleftarrow x$ represents the parent of $x$.  Let $\tilde{x}$ be (the unique) point in $\overline{A\backslash\{x\}}$ such that $d(x,\tilde x)=\min\{d(x,y), \ y\in \overline{A\backslash\{x\}}\}$. I.e., 
\begin{itemize} 
\item If $x$ is a descendant of the common ancestor in $\overline{A\backslash\{x\}}$, then $\tilde{x}$ has to be the nearest ancestor of $x$ in $\overline{A\backslash\{x\}}$.
\item Otherwise, $\tilde{x}$ has to be the common ancestor itself. 
\end{itemize}
 Note that $\bar{A}=\overline{A\backslash\{x\}}\cup\overline{\{x,\tilde{x}\}}$, $\overline{A\backslash\{x\}}\cap\overline{\{x,\tilde{x}\}}=\{\tilde{x}\}$ and $\hat A=\hat {\bar A}$, so that
\begin{equation}\label{eq2}
|\hat A|
\begin{cases}
 =|\widehat{A\backslash\{x\}}|+1&\text{ if } \overleftarrow x\in H( A)\\
 \ge|\widehat{A\backslash\{x\}}|&\text{ otherwise }.
 \end{cases}
 \end{equation}
 Then by the inductive hypothesis and \eqref{eq1} and \eqref{eq2}, the claim holds for the case when $|A|=n+1$ and thus holds for any $n$.
 For any $x\in\hat{A}$, we consider the process $\xi_t^{\lambda,x,\BT_x^d}$. Obviously, $\xi_t^{\lambda,x,\BT_x^d},x\in\hat A$ are identically independent,
so that
\bea\label{H}
P^{\xi_A}(\xi^{\lambda,\BT_d}_t \text{ dies out})&\leq P(\forall x\in\hat A,\xi^{\lambda,x,\BT^d_{x}}_t\text{ dies out})
\\&=\prod_{x\in \hat A}P( \xi^{\lambda,x,\BT^d_{x}}_t\text{ dies out})\\&\leq p_2^{\frac{n}{2}}.\eea

Thus by \eqref{F}, \eqref{H}, and the translation invariance,
\begin{center}$\lim\limits_{n}\sup\limits_{|A|=n}P( \xi^{\lambda,x,\BT^d_{x}}_t\text{ dies out})\leq\lim\limits_n[max\{p_1, p_2\}]^{\frac{n}{2}}=0$\end{center}\end{proof}
\begin{remark}\label{rmk62}
It is easy to verify that all the lemmas above hold if we replace $\xi^{\lambda,\BT_d}_t$ by the threshold-one voter model with spontaneous death we are talking about in this paper. 
\end{remark}


\subsection{Complete convergence theorem}

\label{CCT}

Inspired by Theorem 1.2 of \cite{bramson1991annihilating} and Theorem 1.12 of \cite{liggett1999stochastic}, we will prove the following statement and then prove the complete convergence theorem for the threshold-one contact process, which are also applicable to the threshold-one voter model with positive spontaneous death and only annihilating duality.
\begin{proposition}\label{lem57}
For any $\mu$ satisfying $\mu(\xi=\textbf{0})=0$, $\mu$ is translation invariant or $\mu=\delta_\xi$ where $\xi$ is dense, we have
$$
\lim_{t\to\infty}{\mu} S^\lambda(t)=\bar \nu^\lambda,
$$where $S^\lambda(t)$ is the Markov semigroup and $\bar \nu^\lambda$ is the upper invariant measure of ${\xi}^{\lambda,\BT_d}_t$.
\end{proposition}
 \begin{proof}
 
  Let 
  \begin{center}$\Omega_{\infty}=\{\tilde{\xi}^{\lambda,\BT_d}_t\ne\textbf{0}$ for all $t $$\}$.\end{center}
By \eqref{21}, for any $\xi_A$,
$$
\bar{\nu}^\lambda(\xi:\xi\cap \xi_A\ne\textbf{0})=\tilde{P}^{\xi_A}(\Omega_{\infty}).
$$

Thus it suffices to prove that for any $A\subseteq\BT_d$,
 \bea\label{cd}P^{\mu,\xi_A}( \xi^{\lambda,\BT_d}_{2}\cap\tilde{\xi}_t^{\lambda,\BT_d}\ne\textbf{0})\to\tilde{P}^{\xi_A}(\Omega_{\infty}).\eea
when $\tilde{P}^{\xi_A}(\Omega_{\infty})>0$ otherwise the convergence becomes obvious, where $P^{\mu,\xi_A}$ is the coupling measure of $\xi^{\lambda,\BT_d}_{2}$ and $\tilde{\xi}_t^{\lambda,\BT_d}$ starting from $\mu$ and $\delta_{\xi_A}$, which is constructed by the graph representation.

Firstly, we claim that for any $\epsilon>0 , N$,  we can find $K$ such that for any $|A|\geq K$,\bea\label{cc}P^\mu( |\xi^{\lambda,\BT_d}_1\cap A|\leq N )<\epsilon.\eea

When $\mu=\delta_\xi$ and $\xi$ is dense, denote $\CI_x=\{\xi, \xi\cap B_x(N)\ne\textbf{0}\}$, 
we have
 \bea\label{5111}\inf\limits_xP^\xi(\xi_1^{\lambda,\BT_d}(x)=1)
 &\geq\inf\limits_{x}\inf\limits_{\xi\in\CI_x}P^\xi(\xi_1^{\lambda,\BT_d}(x)=1)
 \\&=\inf\limits_{\xi\in\CI_o}P^\xi(\xi_1^{\lambda,\BT_d}(o)=1)\quad\text{ by translation invariance}
 \\&=\inf\limits_{\xi\in\CI_o}P^o(\tilde{\xi}^{\lambda,\BT_d}_1\cap\xi\text{ $\ne\textbf{0}$})\quad\text{ by duality}
 \\&\geq \inf\limits_{\xi\in\CI_o}P^o(\tilde{\xi}^{\lambda,\BT_d}_1\subset B_N(o)\text{ and }\tilde{\xi}^{\lambda,\BT_d}_1 \cap\xi\text{ $\ne\textbf{0}$})
\\&=\inf\limits_{\xi\subseteq B_N(o),\xi\in\CI_o}P^o(\tilde{\xi}^{\lambda,\BT_d}_1\subseteq B_N(o),\tilde{\xi}^{\lambda,\BT_d}_1 \cap\xi\text{ $\ne\textbf{0}$})
\\&=p>0\eea where the last equality holds because that there are finite configurations satisfying $\xi\subseteq B(N)$.

By Theorem 4.6 of Chapter I of \cite{liggett1985}, for any $A\in Y$, there exists $\lambda=\lambda(diam(A))$ where $diam(A)=\max\{|x-y|:x,y\in A\}$ and $\lambda(k)\to 0$ as $k\to\infty$ such that
\bea\label{5100}
E^{\xi}\prod\limits_{x\in A}[1-\xi^{\lambda,\BT_d}_1(x)]\le \prod\limits_{x\in A}E^{\xi}[1-\xi^{\lambda,\BT_d}_1(x)]+\lambda|A|.
\eea
It follows from \eqref{5111} and \eqref{5100} that
\bea\label{dirac}
P^\xi( \xi_1^{\lambda,\BT_d}\cap \xi_A=\textbf{0} )
&=E^{\xi}\prod\limits_{x\in A}[1-\xi^{\lambda,\BT_d}_1(x)]
\\&\le\prod\limits_{x\in A}P^\xi(\xi_1^{\lambda,\BT_d}(x)=0)+\lambda|A|
\\&\leq (1-p)^{|A|}+\lambda|A|.
\eea

When $\mu$ is translation invariant, by \eqref{5100} and the AM-GM inequality, we have
\bea\label{tran}
E^{\mu}\prod\limits_{x\in A}[1-\xi^{\lambda,\BT_d}_1(x)]
&\le \int\prod\limits_{x\in A}E^{\xi}[1-\xi^{\lambda,\BT_d}_1(x)]\mu(d\xi)+\lambda|A|
\\&\le E^{\mu}[1-\xi^{\lambda,\BT_d}_1(o)]^n+\lambda|A|,
\eea
where the first term converges to 0 when $n$ converges to $\infty$.

Then the claim \eqref{cc} follows from \eqref{dirac} and \eqref{tran}.
Since $|\tilde{\xi}_t^{\lambda,\BT_d}|\rightarrow\infty$ in probability on $\Omega_{\infty}$, by a simple calculation, we have
\bea\label{590}
|\xi_1^{\lambda,\BT_d}\cap\tilde{\xi}_t^{\lambda,\BT_d}|\rightarrow\infty \text{ in probability under $P^{\mu,\xi_A}$ on $\Omega_{\infty}$}.
\eea

Next, let \begin{center}$U_t=\xi_1^{\lambda,\BT_d}\cap\tilde{\xi}_t^{\lambda,\BT_d},
V_t=\{x:x\in U_t, \text{ $x$ is isolated}\}$\end{center}
where $x$ is isolated iff $T_n^{x,\CN_x}\cap(\{x\}\times[1,2])=\emptyset$ and $T_n^{y,\CN_y}\cap(\{y\}\times[1,2])=\emptyset$ for all $x\in \CN_y$. 
It makes sure that
 \bea\label{53}
|\xi_{2}^{\lambda,\BT_d}\cap\tilde{\xi}^{\lambda,\BT_d}_t\cap V_t|=\sum\limits_{x\in V_t}1_{\{T_n^{x,\emptyset}\cap (\{x\}\times[1,2])=\emptyset\}}.
\eea

Since the events $\{x\in V_t\}_{x\in U_t}$  are 2-dependent(i.e. independent when the distance is larger than 2), by \eqref{590},
\bea\label{V}
|V_t|\to\infty  \text{ in probability under $P^{\mu,\xi_A}$ on }\Omega_{\infty}.\eea

Let 
\bea\label{54}
g_x=1_{\{T_n^{x,\CN_x}\cap (\{x\}\times[1,2])=\emptyset\}}.
\eea
Given $ \CG_t=\sigma(V_t)$, $\{g_x, x\in V_t\}$ are identically independent and independent of $V_t$, moreover, each has positive probability $p_0$ to be 1. By 
\eqref{53}, \eqref{V} and \eqref{54},
\bea\label{551}
1\ge P^{\mu,\xi_A}( \xi_{2}^{\lambda,\BT_d}\cap\tilde{\xi}^{\lambda,\BT_d}_t\text{ $\ne\textbf{0}$ } |\CG_t)\ge P^{\mu,\xi_A}(\sum\limits_{x\in V_t}g_x\ge 1|\CG_t)=P(B(n,p_0)\ge1)|_{n=|V_t|}\to 1
\eea in probability on $\Omega_{\infty}$.
The interpretation is that when $|V_t|$ is large enough, we can find at least one vertex in $V_t$ whose state does not change in the unit time. 
  
  Taking expectation of \eqref{551} and by bounded convergence theorem,
  \begin{align*}P^{\mu,\xi_A}( \xi_{2}^{\lambda,\BT_d}\cap\tilde{\xi}^{\lambda,\BT_d}_t\text{ $\ne\textbf{0}$} |\Omega_{\infty})\to1,
 \end{align*}
which implies \eqref{cd} and hence the proof is complete.
 
 \end{proof}

{\bf Proof of Theorem \ref{complete convergence}:}
\begin{proof}
 When $\lambda\le\lambda_1(d)$ we can get the result by Lemma \ref{le33}.

 When $\lambda_1(d)<\lambda\le\lambda_2(d)$ we can get the result by Proposition \ref{lem57}.

Therefore we only need to prove the result when $\lambda>\lambda_2(d)$. By Lemma \ref{corollary_equivalent}, $\inf_tP^o(\xi_{t}^{\lambda,\BT_d}(o)=1)>0$, so that $\tilde\xi_{t}^{\lambda,\BT_d}$ survive strongly.

Define the stoping times
\bea
&\tau_{B(n)}=\inf\{t\ge0:B(n)\subset \xi^{\lambda,\BT_d}_t\},
\\&\tilde{\tau}_{B(n)}=\inf\{t\ge0:B(n)\subset \tilde{\xi}^{\lambda,\BT_d}_t\}.
\eea
 By Levy's 0-1 Law, we can prove that $\{\xi_{t}^{\lambda,\BT_d}\text{ survives strongly}\}\subseteq{\Omega}_A$ and $\{\tilde\xi_{t}^{\lambda,\BT_d}\text{ survives strongly}\}\subseteq\tilde{\Omega}_A$ a.s., which implies that for any finite $B$,
\bea\label{550}
&P^{\xi_A}(\xi_{t}^{\lambda,\BT_d}\cap\xi_B\text{ $\ne\textbf{0}$},\Omega^c_{B(n)})\to 0,
\\&\tilde{P}^{\xi_{A}}(\tilde{\xi}^{\lambda,\BT_d}_{t}\cap \xi_{B}\text{ $\ne\textbf{0}$},\tilde{\Omega}_{B(n)}^c)\to 0
\eea where $\Omega_A=\{\tau_{A}<\infty\}$ and
$\tilde{\Omega}_A=\{\tilde{\tau}_{A}<\infty\}$.

For any $A$, finite $B$,
\bea\label{533}
&\lim\limits_{t\to\infty}P^{ \xi_A}( \xi_{t+1}^{ \lambda,\BT_d}\cap \xi_B\text{  $\ne\textbf{0}$})
\\&=\lim\limits_{t\to\infty}P^{ \xi_A}( \xi_{t+1}^{ \lambda,\BT_d}\cap \xi_B\text{  $\ne\textbf{0}$},\Omega_{B(n)})\quad\text{by \eqref{550}}
\\&=\lim\limits_{t\to\infty}E^{ \xi_A}[P^{ \xi^{ \lambda,\BT_d}_{\tau_{B(n)}}}[ \xi^{ \lambda,\BT_d}_{t+1-\tau_{B(n)}}\cap \xi_B\text{  $\ne\textbf{0}$}],\Omega_{B(n)}]\quad\text{by strong Markov property}
\\&=\lim\limits_{t\to\infty}E^{ \xi_A}[\tilde{P}^{ \xi_B}[\tilde{ \xi}^{ \lambda,\BT_d}_{t+1-\tau_{B(n)}}\cap  \xi^{ \lambda,\BT_d}_{\tau_{B(n)}}\text{  $\ne\textbf{0}$}],\Omega_{B(n)}]\quad\text{by \eqref{20}}
\\&=\lim\limits_{t\to\infty}E^{ \xi_A}[\tilde{P}^{ \xi_B}[\tilde{ \xi}^{ \lambda,\BT_d}_{t+1-\tau_{B(n)}}\cap  \xi^{ \lambda,\BT_d}_{\tau_{B(n)}}\text{  $\ne\textbf{0}$},\tilde{\Omega}_{B(n)}],\Omega_{B(n)}]\quad\text{by \eqref{550} }
\\&=\lim\limits_{t\to\infty}E^{ \xi_A}[\tilde{E}^{ \xi_B}[\tilde{P}^{\tilde{ \xi}^{ \lambda,\BT_d}_{\tilde{\tau}_{B(n)}}}[\tilde{ \xi}^{ \lambda,\BT_d}_{t+1-\tau_{B(n)}-\tilde{\tau}_{B(n)}}\cap  \xi^{ \lambda,\BT_d}_{\tau_{B(n)}}\text{  $\ne\textbf{0}$}],\tilde{\Omega}_{B(n)}],\Omega_{B(n)}]\quad\text{by strong Markov property}
\\&=\lim\limits_{t\to\infty}E^{ \xi_A}[\tilde{E}^{ \xi_B}[P^{ \xi^{ \lambda,\BT_d}_{\tau_{B(n)}}}[{ \xi}^{ \lambda,\BT_d}_{t+1-\tau_{B(n)}-\tilde{\tau}_{B(n)}}\cap\tilde{ \xi}^{ \lambda,\BT_d}_{\tilde{\tau}_{B(n)}} \text{  $\ne\textbf{0}$}],\tilde{\Omega}_{B(n)}],\Omega_{B(n)}] \quad\text{by \eqref{20}}.
\eea

When $t>\tau_{B(n)}+\tilde{\tau}_{B(n)}$,
\bea\label{525}
&|P^{ \xi^{ \lambda,\BT_d}_{\tau_{B(n)}}}( \xi^{ \lambda,\BT_d}_{t+1-\tau_{B(n)}-\tilde{\tau}_{B(n)}}\cap \tilde{ \xi}^{ \lambda,\BT_d}_{\tilde{\tau}_{B(n)}}\text{  $\ne\textbf{0}$})-1|
\\&\le P^{ \xi^{ \lambda,\BT_d}_{\tau_{B(n)}}}(| \xi^{ \lambda,\BT_d}_{t-\tau_{B(n)}-\tilde{\tau}_{B(n)}}\cap \tilde{ \xi}^{ \lambda,\BT_d}_{\tilde{\tau}_{B(n)}}|\ge k)|P^{ \xi^{ \lambda,\BT_d}_{\tau_{B(n)}}}( \xi^{ \lambda,\BT_d}_{t+1-\tau_{B(n)}-\tilde{\tau}_{B(n)}}\cap \tilde{ \xi}^{ \lambda,\BT_d}_{\tilde{\tau}_{B(n)}}\text{  $\ne\textbf{0}$}|| \xi^{ \lambda,\BT_d}_{t-\tau_{B(n)}-\tilde{\tau}_{B(n)}}\cap \tilde{ \xi}^{ \lambda,\BT_d}_{\tilde{\tau}_{B(n)}}|\ge k)-1|
\\&+2P^{ \xi^{ \lambda,\BT_d}_{\tau_{B(n)}}}(| \xi^{ \lambda,\BT_d}_{t-\tau_{B(n)}-\tilde{\tau}_{B(n)}}\cap \tilde{ \xi}^{ \lambda,\BT_d}_{\tilde{\tau}_{B(n)}}|\le k)
\\&\le |P^{ \xi^{ \lambda,\BT_d}_{\tau_{B(n)}}}( \xi^{ \lambda,\BT_d}_{t+1-\tau_{B(n)}-\tilde{\tau}_{B(n)}}\cap \tilde{ \xi}^{ \lambda,\BT_d}_{\tilde{\tau}_{B(n)}}\text{  $\ne\textbf{0}$}|| \xi^{ \lambda,\BT_d}_{t-\tau_{B(n)}-\tilde{\tau}_{B(n)}}\cap \tilde{ \xi}^{ \lambda,\BT_d}_{\tilde{\tau}_{B(n)}}|\ge k)-1|
\\&+2P^{ \xi^{ \lambda,\BT_d}_{\tau_{B(n)}}}(| \xi^{ \lambda,\BT_d}_{t-\tilde{\tau}_{B(n)}-\tau_{B(n)}}\cap B(n)|\le k)
\eea
for any $n,k$.

Let $U_t=\xi^{\lambda,\BT_d}_{t-\tau_{B(n)}-\tilde{\tau}_{B(n)}}\cap \tilde{\xi}^{\lambda,\BT_d}_{\tilde{\tau}_{B(n)}}$ and $V_t$ be the corresponding variable, by  \eqref{551},
for any $\epsilon>0$,
\bea\label{527}
&|P^{\xi_{\tau_{B(n)}}}(\xi^{\lambda,\BT_d}_{t+1-\tau_{B(n)}-\tilde{\tau}_{B(n)}}\cap \tilde{\xi}^{\lambda,\BT_d}_{\tilde{\tau}_{B(n)}}\text{ $\ne\textbf{0}$}||\xi^{\lambda,\BT_d}_{t-\tau_{B(n)}-\tilde{\tau}_{B(n)}}\cap \tilde{\xi}^{\lambda,\BT_d}_{\tilde{\tau}_{B(n)}}|\ge k)-1|<\epsilon/2
\eea  for any $t,n$ when $k$ is large enough.

By Lemma \ref{lem53} and Remark \ref{rmk62}, for any $k$,
\bea\label{526}
P^{\xi_{B(n)}}(|\xi_{t-\tau_{B(n)}-\tilde{\tau}_{B(n)}}\cap B(n)|\le k)<\epsilon/2
\eea  for any $t>\tau_{B(n)}+\tilde{\tau}_{B(n)}$ when $n$ is large enough.

Thus for any $\epsilon>0$, we can find large enough $t_0,n_0$ such that for any $t\ge t_0,n\ge n_0$,
$$
|P^{\xi_A}(\xi_{t+1}^{\lambda,\BT_d}\cap\xi_B\text{ $\ne\textbf{0}$})-P^{\xi_A}(\Omega_{B(n)})\tilde{P}^{\xi_B}(\tilde{\Omega}_{B(n)})|<\epsilon.
$$
Let $\epsilon\to 0$, we can get
\bea\label{534}
\lim\limits_{t\to\infty}P^{\xi_A}(\xi_{t+1}^{\lambda,\BT_d}\cap\xi_B\text{ $\ne\textbf{0}$})&=\lim\limits_{n\to\infty}P^{\xi_A}(\Omega_{B(n)})\tilde{P}^{\xi_B}(\tilde{\Omega}_{B(n)}).
\eea
 We claim that
\bea\label{535}
\lim\limits_{n\to\infty}P^{\xi_A}(\Omega_{B(n)})=P^{\xi_A}(o\in\xi_t^{\lambda,\BT_d}\text{ i.o.}).
\eea
Since on one hand,
\bea\label{544}
P^{\xi_A}(\tau_{B(n)}<\infty,o\in\xi_t^{\lambda,\BT_d}\text{ f.o.})&=P^{\xi_A}[P^{\xi_{\tau_{B(n)}}}[o\in\xi_t^{\lambda,\BT_d}\text{ f.o.}],\tau_{B(n)}<\infty]
\\&\le P^{\xi_A}[P^{\xi_{\tau_{B(n)}}}[o\in\xi_t^{\lambda,\BT_d}\text{ f.o.}]]
\\&=P^{\xi_A}[P^{\xi_{\tau_{B(n)}}}[\tau_{\textbf{0}}<\infty]]
\\&\le P^{\xi_A}[P^{B(n)}[\tau_{\textbf{0}}<\infty]]
\\&\to 0
\eea
where the third equality is the equivalence of the survival and strong survival and the last term is from Lemma \ref{lem44}.  On the other hand, by Levy's 0-1 Law,
\bea\label{56}
P^{\xi_A}(\tau_{B(n)}<\infty|o\in\xi_t^{\lambda,\BT_d}\text{ i.o.})=1.
\eea
Then \eqref{535} comes from \eqref{544} and \eqref{56}.

Similarly, we can get
\bea\label{hhhh}
\lim\limits_{n\to\infty}\tilde P^{\xi_B}(\tilde \Omega_{B(n)})=\tilde P^{\xi_B}(o\in\tilde\xi_t^{\lambda,\BT_d}\text{ i.o.}).
\eea

By \eqref{534}, \eqref{535}, \eqref{hhhh} and the equivalence of survival and strong survival when $\xi_t^{\lambda,\BT_d}$ survives strongly,
\bea\label{536}
\lim\limits_{t\to\infty}P^{\xi_A}(\xi_{t+1}^{\lambda,\BT_d}\cap\xi_B\text{ $\ne\textbf{0}$})=P^{\xi_A}(o\in\xi_t^{\lambda,\BT_d}\text{ i.o.})P^{\xi_B}(o\in\tilde\xi_t^{\lambda,\BT_d}\text{ i.o.})=P^{\xi_A}(\xi_t^{\lambda,\BT_d}\ne\textbf{0},\forall t)\tilde{P}^{\xi_B}(\tilde\xi_t^{\lambda,\BT_d}\ne\textbf{0},\forall t).\eea
By \eqref{21},
\bea
\lim\limits_{t\to\infty}P^{\xi_A}(\xi_{t+1}^{\lambda,\BT_d}\cap\xi_B\text{ $\ne\textbf{0}$})&=P^{\xi_A}(\xi_t^{\lambda,\BT_d}\ne\textbf{0},\forall t)\bar \nu^\lambda(\xi:\xi_B\cap\xi\ne\textbf{0})
\\&=\delta_{\textbf{0}}(\xi:\xi_B\cap\xi\ne\textbf{0})+P^{\xi_A}(\xi_t^{\lambda,\BT_d}\ne\textbf{0},\forall t)\bar \nu^\lambda(\xi:\xi_B\cap\xi\ne\textbf{0}),
\eea
which implies \eqref{ccc} and the proof is complete.
\end{proof}


\section{The threshold-one voter model with positive spontaneous death }

Now we use the same method in the proof of Proposition \ref{lem57} and Theorem \ref{complete convergence} to prove the corresponding Proposition \ref{lem7} and Theorem \ref{thm62} for the threshold-one voter model with positive spontaneous death. As thus, the proofs here are very similar to those in Section \ref{CCT} except for some details, it may be a bit tedious.
\begin{proposition}\label{lem7}
For any $\mu$ satisfying $\mu(\eta=\textbf{0})=0$, $\mu$ is  translation invariant or $\mu=\delta_\eta$ where $\eta$ is dense, we have
$$
\lim_{t\to\infty}{\mu} S^\varepsilon(t)=\bar \nu^\varepsilon,
$$where $S^\varepsilon(t)$ is the Markov semigroup and $\bar \nu^\varepsilon$ is the maximal invariant measure of ${\eta}^{ \varepsilon,\BT_d}_t$.

\end{proposition}
 \begin{proof}
 
  Let 
  \begin{center}$\Omega_{\infty}=\{\tilde{\eta}^{ \varepsilon,\BT_d}_t\ne\emptyset$ for all $t $$\}$.\end{center}
By annihilating duality \eqref{2_cd}, 
\begin{center}$P^{\nu_{1/2}}(\eta_{t}^{ \varepsilon,\BT_d}\cap\eta_A\text{ is odd } )=\tilde{P}^{\eta_A}( \tilde{\eta}^{ \varepsilon,\BT_d}_{t}\cap{\nu_{1/2}}\text{ is odd } )=\frac{1}{2}\tilde P^{\eta_A}(\tilde{\eta}^{ \varepsilon,\BT_d}_t\ne\textbf{0})\to\frac{1}{2}\tilde P^{\eta_A}(\Omega_{\infty})$.\end{center}
Then it suffices to prove that for any $A\subseteq\BT_d$,
 \bea\label{6_1}P^{\mu,\eta_A}( \eta_{2}^{ \varepsilon,\BT_d}\cap\tilde\eta_t^{\BT_d}\text{ is odd })\to\frac{1}{2}\tilde P^{\eta_A}(\Omega_{\infty})\eea
when $\tilde P^{\eta_A}(\Omega_{\infty})>0$.
Since the dual process has positive death rate, it is easy to verify that $|\tilde\eta^{ \varepsilon,\BT_d}_t|\to\infty$ on $\Omega_\infty$. Then applying the proof of \eqref{590} in Proposition \ref{lem57}, we can get
\begin{center}$|\eta_1^{ \varepsilon,\BT_d}\cap\tilde{\eta}_t^{ \varepsilon,\BT_d}|\rightarrow\infty$ in probability under $P^{\mu,\eta_A}$ on $\Omega_{\infty}$.\end{center}

Let \begin{center}$U_t=\eta_1^{ \varepsilon,\BT_d}\cap\tilde{\eta}_t^{ \varepsilon,\BT_d},
V_t=\{x:x\in U_t, \text{ $x$ is isolated}\}$\end{center}
where $x$ is isolated iff $T_n^{x,S_x}\cap(\{x\}\times[1,2])=\emptyset$ for all $S_x$ and $T_n^{y,S_y}\cap(\{y\}\times[1,2])=\emptyset$ for all $x\in S_y$. 
It makes sure that
 \bea\label{3}
|\eta_{2}^{ \varepsilon,\BT_d}\cap\tilde{\eta}^{ \varepsilon,\BT_d}_t\cap V_t|=\sum\limits_{x\in V_t}1_{\{T_n^{x,\varepsilon}\cap (\{x\}\times[1,2])=\emptyset\}}.
\eea

Since the events $\{x\in V_t\}_{x\in U_t}$  are 2-dependent,
\bea\label{670}
|V_t|\to\infty  \text{ in probability under $P^{\mu,\eta_A}$ on }\Omega_{\infty}.
\eea
Let 
\bea\label{4}
&g_x=1_{T_n^{x,\varepsilon}\cap (\{x\}\times[1,2])=\emptyset},
\\&h=1-\{|\eta_{2}^{ \varepsilon,\BT_d}\cap\tilde{\eta}^{ \varepsilon,\BT_d}_t\cap V^c_t| \text{ mod 2}\},
\\&\CG_t=\sigma(\tilde{\eta}_t^{ \varepsilon,\BT_d}, \eta_1^{ \varepsilon,\BT_d},V_t,h).
\eea
Since given $ \CG_t$, $h$ is constant and $\{g_x, x\in V_t\}$ are identically independent and independent of $V_t$, by 
\eqref{3}, \eqref{670} and \eqref{4}, Lemma 2.3 of \cite{bramson1991annihilating},
\bea\label{24}
|P^{(\mu,\eta_A)}( \eta_{2}^{ \varepsilon,\BT_d}\cap\tilde{\eta}^{ \varepsilon,\BT_d}_t\text{ is odd } |\CG_t)-\frac{1}{2}|=|P^{(\mu,\eta_A)}(\sum\limits_{x\in V_t}g_x= h\text{ mod 2}|\CG_t)-\frac{1}{2}|\leq|1-2e^{-\varepsilon}|^{|V_t|}\to 0
\eea in probability on $\Omega_{\infty}$. The interpretation is that when $|V_t|$ is large enough, with half probability, we can find odd number of vertices in $V_t$ whose states do not change in the unit time. 
  
  Taking expectation and by bounded convergence theorem,
  
  \begin{align*}P^\mu( \eta_{2+t}\cap\eta_A\text{ is odd } )=P^{(\mu,\eta_A)}( \eta_{2}^{ \varepsilon,\BT_d}\cap\tilde{\eta}^{ \varepsilon,\BT_d}_t\text{ is odd } )\to\frac{1}{2}\tilde P^{\eta_A}(\Omega_{\infty}),
 \end{align*}
 so that \eqref{6_1} is obtained and the proof is complete.
 \end{proof}

{\bf Proof of Theorem \ref{thm62}:}
\begin{proof}
When $ \varepsilon\ge \varepsilon_1(d)$ we can get the result directly.

 When $ \varepsilon_1(d)>\varepsilon\ge \varepsilon_2(d)$ we can get the result by Proposition \ref{lem7}.

Therefore we only need to prove the result when $ \varepsilon< \varepsilon_2(d)$, by Lemma \ref{corollary_equivalent} and Remark \ref{rmk62}, $\inf_tP^o( \eta_{t}^{ \varepsilon,\BT_d}(o)=1)>0$, so that the dual process $\tilde \eta_{t}^{ \varepsilon,\BT_d}$ survive strongly.

Define the stoping times
\bea
&\tau_{B(n)}=\inf\{t\ge0:B(n)\subset  \eta^{ \varepsilon,\BT_d}_t\},
\\&\tilde{\tau}_{B(n)}=\inf\{t\ge0:B(n)\subset \tilde{ \eta}^{ \varepsilon,\BT_d}_t\}.
\eea
 By Levy's 0-1 Law, we can prove that $\{ o\in\eta_{t}^{ \varepsilon,\BT_d}i.o.\}\subseteq{\Omega}_A$ and $\{o\in\tilde \eta_{t}^{ \varepsilon,\BT_d}i.o.\}\subseteq\tilde{\Omega}_A$ a.s., which implies that for any finite $B$,
\bea\label{650}
&P^{ \eta_A}( \eta_{t}^{ \varepsilon,\BT_d}\cap \eta_B\text{  is odd},\Omega^c_{B(n)})\to 0,
\\&\tilde{P}^{ \eta_{A}}(\tilde{ \eta}^{ \varepsilon,\BT_d}_{t}\cap  \eta_{B}\text{  is odd},\tilde{\Omega}_{B(n)}^c)\to 0
\eea where $\Omega_A=\{\tau_{A}<\infty\}$ and
$\tilde{\Omega}_A=\{\tilde{\tau}_{A}<\infty\}$.

For any $A$, finite $B$, 
\bea\label{633}
&\lim\limits_{t\to\infty}P^{ \eta_A}( \eta_{t+1}^{ \varepsilon,\BT_d}\cap \eta_B\text{  is odd})
\\&=\lim\limits_{t\to\infty}P^{ \eta_A}( \eta_{t+1}^{ \varepsilon,\BT_d}\cap \eta_B\text{  is odd},\Omega_{B(n)})\quad\text{by \eqref{650}}
\\&=\lim\limits_{t\to\infty}E^{ \eta_A}[P^{ \eta^{ \varepsilon,\BT_d}_{\tau_{B(n)}}}[ \eta^{ \varepsilon,\BT_d}_{t+1-\tau_{B(n)}}\cap \eta_B\text{  is odd}],\Omega_{B(n)}]\quad\text{by strong Markov property}
\\&=\lim\limits_{t\to\infty}E^{ \eta_A}[\tilde{P}^{ \eta_B}[\tilde{ \eta}^{ \varepsilon,\BT_d}_{t+1-\tau_{B(n)}}\cap  \eta^{ \varepsilon,\BT_d}_{\tau_{B(n)}}\text{  is odd}],\Omega_{B(n)}]\quad\text{by \eqref{2_cd}}
\\&=\lim\limits_{t\to\infty}E^{ \eta_A}[\tilde{P}^{ \eta_B}[\tilde{ \eta}^{ \varepsilon,\BT_d}_{t+1-\tau_{B(n)}}\cap  \eta^{ \varepsilon,\BT_d}_{\tau_{B(n)}}\text{  is odd},\tilde{\Omega}_{B(n)}],\Omega_{B(n)}]\quad\text{by \eqref{650} }
\\&=\lim\limits_{t\to\infty}E^{ \eta_A}[\tilde{E}^{ \eta_B}[\tilde{P}^{\tilde{ \eta}^{ \varepsilon,\BT_d}_{\tilde{\tau}_{B(n)}}}[\tilde{ \eta}^{ \varepsilon,\BT_d}_{t+1-\tau_{B(n)}-\tilde{\tau}_{B(n)}}\cap  \eta^{ \varepsilon,\BT_d}_{\tau_{B(n)}}\text{  is odd}],\tilde{\Omega}_{B(n)}],\Omega_{B(n)}]\quad\text{by strong Markov property}
\\&=\lim\limits_{t\to\infty}E^{ \eta_A}[\tilde{E}^{ \eta_B}[P^{ \eta^{ \varepsilon,\BT_d}_{\tau_{B(n)}}}[{ \eta}^{ \varepsilon,\BT_d}_{t+1-\tau_{B(n)}-\tilde{\tau}_{B(n)}}\cap\tilde{ \eta}^{ \varepsilon,\BT_d}_{\tilde{\tau}_{B(n)}} \text{  is odd}],\tilde{\Omega}_{B(n)}],\Omega_{B(n)}] \quad\text{by \eqref{2_cd}}.
\eea

Fix $ \eta^{ \varepsilon,\BT_d}_{\tau_{B(n)}}$, $\tilde{ \eta}^{ \varepsilon,\BT_d}_{\tilde{\tau}_{B(n)}}$, $\tau_{B(n)}$ and $\tilde\tau_{B(n)}$, when $t>\tau_{B(n)}+\tilde{\tau}_{B(n)}$,
\bea\label{625}
&|P^{ \eta^{ \varepsilon,\BT_d}_{\tau_{B(n)}}}( \eta^{ \varepsilon,\BT_d}_{t+1-\tau_{B(n)}-\tilde{\tau}_{B(n)}}\cap \tilde{ \eta}^{ \varepsilon,\BT_d}_{\tilde{\tau}_{B(n)}}\text{  is odd})-\frac{1}{2}|
\\&\le|P^{ \eta^{ \varepsilon,\BT_d}_{\tau_{B(n)}}}(| \eta^{ \varepsilon,\BT_d}_{t-\tau_{B(n)}-\tilde{\tau}_{B(n)}}\cap \tilde{ \eta}^{ \varepsilon,\BT_d}_{\tilde{\tau}_{B(n)}}|\ge k)P^{ \eta^{ \varepsilon,\BT_d}_{\tau_{B(n)}}}( \eta^{ \varepsilon,\BT_d}_{t+1-\tau_{B(n)}-\tilde{\tau}_{B(n)}}\cap \tilde{ \eta}^{ \varepsilon,\BT_d}_{\tilde{\tau}_{B(n)}}\text{  is odd}|| \eta^{ \varepsilon,\BT_d}_{t-\tau_{B(n)}-\tilde{\tau}_{B(n)}}\cap \tilde{ \eta}^{ \varepsilon,\BT_d}_{\tilde{\tau}_{B(n)}}|\ge k)-\frac{1}{2}|
\\&+\frac{3}{2}P^{ \eta^{ \varepsilon,\BT_d}_{\tau_{B(n)}}}(| \eta^{ \varepsilon,\BT_d}_{t-\tau_{B(n)}-\tilde{\tau}_{B(n)}}\cap \tilde{ \eta}^{ \varepsilon,\BT_d}_{\tilde{\tau}_{B(n)}}|\le k)
\\&\le |P^{ \eta^{ \varepsilon,\BT_d}_{\tau_{B(n)}}}( \eta^{ \varepsilon,\BT_d}_{t+1-\tau_{B(n)}-\tilde{\tau}_{B(n)}}\cap \tilde{ \eta}^{ \varepsilon,\BT_d}_{\tilde{\tau}_{B(n)}}\text{  is odd}|| \eta^{ \varepsilon,\BT_d}_{t-\tau_{B(n)}-\tilde{\tau}_{B(n)}}\cap \tilde{ \eta}^{ \varepsilon,\BT_d}_{\tilde{\tau}_{B(n)}}|\ge k)-1|
\\&+\frac{3}{2}P^{ \eta^{ \varepsilon,\BT_d}_{\tau_{B(n)}}}(| \eta^{ \varepsilon,\BT_d}_{t-\tilde{\tau}_{B(n)}-\tau_{B(n)}}\cap B(n)|\le k)
\eea
for any $n,k$.

Let $U_t= \eta^{ \varepsilon,\BT_d}_{t-\tau_{B(n)}-\tilde{\tau}_{B(n)}}\cap \tilde{ \eta}^{ \varepsilon,\BT_d}_{\tilde{\tau}_{B(n)}}$ and $V_t$ be the corresponding variable, by  \eqref{24}, for any $\epsilon>0$,
\bea\label{627}
&|P^{ \eta^{ \varepsilon,\BT_d}_{\tau_{B(n)}}}( \eta^{ \varepsilon,\BT_d}_{t+1-\tau_{B(n)}-\tilde{\tau}_{B(n)}}\cap \tilde{ \eta}^{ \varepsilon,\BT_d}_{\tilde{\tau}_{B(n)}}\text{  is odd}|| \eta^{ \varepsilon,\BT_d}_{t-\tau_{B(n)}-\tilde{\tau}_{B(n)}}\cap \tilde{ \eta}^{ \varepsilon,\BT_d}_{\tilde{\tau}_{B(n)}}|\ge k)-\frac{1}{2}|
<\epsilon/2
\eea  uniformly in $t,n$ when $k$ is large enough.

By Lemma \ref{lem53} and Remark \ref{rmk62}, for any $k$,
\bea\label{626}
P^{ \eta_{B(n)}}(| \eta^{ \varepsilon,\BT_d}_{t-\tau_{B(n)}-\tilde{\tau}_{B(n)}}\cap B(n)|\le k)<\epsilon/2
\eea  uniformly in $t>\tau_{B(n)}-\tilde{\tau}_{B(n)}$ when $n$ is large enough.

Thus for any $\epsilon>0$, we can find large enough $t_0,k_0$ such that for any $t\ge t_0,n\ge n_0$,

$$
|P^{ \eta_A}( \eta_{t+1}^{ \varepsilon,\BT_d}\cap \eta_B\text{  is odd})-\frac{1}{2}|<\epsilon.
$$

Let $\epsilon\to 0$, we can get
\bea\label{634}
\lim\limits_{t\to\infty}P^{ \eta_A}( \eta_{t+1}^{ \varepsilon,\BT_d}\cap \eta_B\text{  is odd})=\frac{1}{2}\lim\limits_{n\to\infty}P^{ \eta_A}(\Omega_{B(n)})\tilde{P}^{ \eta_B}(\tilde{\Omega}_{B(n)}).
\eea
 
Since $\eta_t^{ \varepsilon,\BT_d}$ is attractive while $\tilde\eta_t^{ \varepsilon,\BT_d}$ is not, applying the proof of 
\eqref{535}, we can only get
$$
\lim\limits_{n\to\infty}P^{ \eta_A}(\Omega_{B(n)})=P^{ \eta_A}(o\in \eta_t^{ \varepsilon,\BT_d}\text{ i.o.}).
$$
It follows that for any initial configuration $\eta$, 
\bea\label{8}
\lim\limits_{t\to\infty}P^{\eta}(\eta_{t+1}^{ \varepsilon,\BT_d}\cap\eta_B\text{ is odd})=\frac{1}{2}\lim\limits_{n\to\infty}\tilde{P}^{\eta_B}(\tilde{\Omega}_{B(n)})P^{\eta}(o\in\eta_t^{ \varepsilon,\BT_d}\text{ i.o.}),\eea
while by Proposition \ref{lem7}, for all dense configuration $\eta$,

\bea\label{69}
P^\eta(\eta^{ \varepsilon,\BT_d}_{t}\cap\eta_B\text{ is odd})\to\frac{1}{2}\tilde P^{\eta_B}(\Omega_\infty).
\eea
Thus by \eqref{8} and \eqref{69},
\bea\label{30}
\lim\limits_{n\to\infty}\tilde{P}^{\eta_B}(\tilde{\Omega}_{B(n)})=\tilde P^{\eta_B}({\Omega}_{\infty}).
\eea
Then by \eqref{8}, \eqref{30}, and the equivalence of the survival and strong survival when $\eta_t^{ \varepsilon,\BT_d}$ survives strongly,
$$
P^{\eta_A}(\eta^{ \varepsilon,\BT_d}_{t}\cap\eta_B\text{ is odd})\to\frac{1}{2}P^{\eta_A}(o\in\eta_t^{ \varepsilon,\BT_d}\text{ i.o.})\tilde P^{\eta_B}(\Omega_\infty)=\frac{1}{2} P^{\eta_A}(\eta_t^{ \varepsilon,\BT_d}\ne\emptyset,\forall t)\tilde P^{\eta_B}(\Omega_\infty).
$$
By Proposition \ref{lem7},
$$
P^{\eta_A}(\eta^{ \varepsilon,\BT_d}_{t}\cap\eta_B\text{ is odd})\to P^{\eta_A}(\eta_t^{ \varepsilon,\BT_d}\ne\emptyset,\forall t)\bar\nu^\varepsilon(\eta:\eta\cap\eta_B\text{ is odd}).
$$
\end{proof}

\section{The threshold-one voter model }
\label{section_7}

In this section, we consider the threshold-one voter model $\zeta^{\BT_d}_{t}$ on $\BT_d$, which has no spontaneous death. Our main task is to prove the complete convergence theorem when it survives strongly, while otherwise, only if $\beta(d)< \frac{1}{\sqrt{d}}$ we can prove that starting from any translation invariant measure $\mu$ or Dirac measure concentrating on the doubly dense configuration, the process converges to $\mu_{1/2}$. First of all, we need to prove the following lemmas.

\begin{lemma}\label{lem71}
If $\beta(d)< \frac{1}{\sqrt{d}}$,
\bea
\tilde P^o(o\in\tilde{\zeta}_t^{\BT_d}\text{ i.o.})=0.
\eea
\end{lemma}
\begin{proof}

Similar to the corresponding Theorem 2 in \cite{lalley1999growth} for the ordinary contact process on $\BT_d$, we can also prove that
\bea\label{74}
[\lim\limits_{t\to\infty}P^o(o\in\zeta^{\BT_d}_{t})]^{1/t}=\gamma<1.
\eea
Next we will use the techniques in the proof of Theorem 1.2 of \cite{louidor2014williams}, such that for any $s\in \BN$, let $s^2=t_0<t_1<\cdots<t_{s^4}=(s+1)^2$ be a uniform partition of $[s^s,(s+1)^2)$, then
\bea\label{77}
&\tilde P^o(o\in\tilde{\zeta}^{\BT_d}_t\text{ for some }t\in[s^2,(s+1)^2))
\\&\le \tilde P^o(o\in\tilde{\zeta}^{\BT_d}_{t_k}\text{ for some }0\le k\le s^4)+P(\exists 0\le k\le s^4, |(\cup_{S_o}T_n^{o,S_o})\cap[t_k,t_{k+1})|\ge 2).
\eea
For the first term, by \eqref{74}, we have 
\bea\label{75}
\tilde P^o(o\in\tilde{\zeta}^{\BT_d}_{t_k}\text{ for some }0\le k\le s^4)\le s^4\sup_{s^2\le t\le (s+1)^2}\tilde P^o(o\in \tilde{\zeta}^{\BT_d}_{t})\le C\gamma^{s^2}.
\eea
For the second term, by the property of the Poisson distribution, we have
\bea\label{76}
P(\exists 0\le k\le s^4, |(\cup_{S_o}T_n^{o,S_o})\cap[t_k,t_{k+1})|\ge 2)
\le Cs^4(t_{k+1}-t_k)^2\le Cs^{-2}.
\eea
Substitute \eqref{75} and \eqref{76} in \eqref{77}, we have
$$
\sum\limits_{s\in\BN}\tilde P^o(o\in\tilde{\zeta}^{\BT_d}_t\text{ for some }t\in[s^2,(s+1)^2))<\infty.
$$
Thus by Borel-Cantelli Lemma we can get \eqref{lem71} .

\end{proof}
Denote the equivalent class of $\tilde{\zeta}^{\BT_d}_t$ as $[\tilde{\zeta}^{\BT_d}_t]$, where two configurations are equivalent if and only if one can be translated by the other. 

\begin{lemma}\label{lem72}
If $\beta(d)\ne\frac{1}{\sqrt{d}}$, $[\tilde{\zeta}^{o,\BT_d}_t]$ is not positive recurrent.
\end{lemma}
\begin{proof}
Now we assume that $[\tilde{\zeta}^{\BT_d}_t]$ is positive recurrent, let
$$
\tau_{*}=\inf\{t:[\tilde{\zeta}^{\BT_d}_t]=[\{o\}]\}, \tau_{o}=\inf\{t:o\in{\zeta}^{\BT_d}_t\},\tilde\tau_{o}=\inf\{t:o\in\tilde{\zeta}^{\BT_d}_t\}
$$
 so that we have $\tilde P^o(\tau_{*}<\infty)=1$. 

Firstly when $\beta(d)<\frac{1}{\sqrt{d}}$, we claim that for any $A$ with finite odd number of vertices,
\bea\label{80}
 \tilde P^A(\tilde{\tau}_{o}<\infty)=1.
\eea

Define $\CN_o=\{x_1,\dots,x_{d+1}\}$, and $\BT_i,1\le i\le d+1$ as the connected component of $\BT_d\backslash\{o\}$ containing $x_i$. 

When $o\not\in\bar A$ where $\bar A$ is defined in Section \ref{section_3}, let $\zeta$ satisfy that $\zeta\cap\zeta_{\BT_i}=\zeta_{A_i}$, where $[A_i]=[A]$ for any $1\le i\le 3$ and $A_i=\emptyset$ otherwise, then by the independence of $\eta^{\zeta_{A_i},\BT_d}_{t\land\tau_{o}},1\le i\le d+1$, we have 
\bea\label{prod}
P^\zeta(\tilde\tau_{o}=\infty)=\prod_{i=1}^{d+1}P^{\zeta_{A_i}}(\tilde\tau_{o}=\infty)
\eea
so that by Markov property,
\bea\label{79}
 \tilde P^o(\tau_{*}=\infty)\ge\tilde P^o(\tilde{\zeta}_1=\zeta)[\tilde P^A(\tilde{\tau}_{o}=\infty)]^3.
 \eea
 Since $\tilde P^o(\tilde{\zeta}_1=\zeta)>0$ by irreducibility, it follows from the assumption that
  \bea\label{nin}
\tilde P^A(\tilde{\tau}_{o}=\infty)=0.
\eea 
 
When $o\in\bar{A}$, let $A_{odd}=A\cap\BT_i$ for some $i$ with odd number of vertices, by \eqref{prod} and \eqref{nin}, 
\bea\label{78}
\tilde P^{\zeta_A}(\tilde{\tau}_{o}=\infty)\le \tilde P^{\zeta_{A_{odd}}}(\tilde{\tau}_{o}=\infty)=0
\eea
Thus from \eqref{nin} and \eqref{78}, we can get \eqref{80}.

But if the claim \eqref{80} is true, 
$$
\tilde P^o(o\in\tilde{\zeta}^{\BT_d}_t\text{ i.o.})=1,
$$
which is impossible by Lemma \ref{lem71}, so that the assumption is false.

Next when $\beta(d)>\frac{1}{\sqrt{d}}$, we will use the method of proving Lemma 2.2 of \cite{handjani1999complete} to get the result.
By the assumption,
\bea\label{81}
\lim\limits_{k\to\infty}\lim\limits_{t\to\infty}\tilde P^o(diam([\tilde{\zeta}^{\BT_d}_t])\ge k)= 0.
\eea

Since $\tilde{\zeta}^{\BT_d}_t$ is not positive recurrent, 
\bea\label{82}
\lim\limits_{k\to\infty}\lim\limits_{t\to\infty}\tilde P^o(\tilde{\zeta}^{\BT_d}_t\subseteq B_k(o))= 0.
\eea

Then it follows from \eqref{81} and \eqref{82} that
$$
\lim\limits_{t\to\infty}\tilde P^o(o\in\tilde{\zeta}^{\BT_d}_t)\le \lim\limits_{k\to\infty}\lim\limits_{t\to\infty}\tilde P^o(\tilde{\zeta}^{\BT_d}_t\subseteq B_k(o))+\lim\limits_{k\to\infty}\lim\limits_{t\to\infty}\tilde P^o(diam([\tilde{\zeta}^{\BT_d}_t])\ge k)= 0,
$$

which is impossible by Remark \ref{rmk62} and Lemma \ref{le64}. So that the assumption is also false. Now we have proved that $[\tilde{\zeta}^{\BT_d}_t]$ is not positive recurrent for any $d$ satisfying $\beta(d)\ne\frac{1}{\sqrt{d}}$.
\end{proof}

\begin{lemma}\label{lem73}
If $\beta(d)\ne\frac{1}{\sqrt{d}}$, for any $k>0$, finite $A$,
$$
\tilde P^{\zeta_A}(|\tilde{\zeta}^{\BT_d}_t|=k|\tilde\zeta_t^{\BT_d}\ne\textbf{0},\forall t)\to 0.
$$
\end{lemma}
\begin{proof}
For $k=1$, by Lemma \ref{lem72}, the process is not positive recurrent, so that
$$
\tilde P^{\zeta_A}(|\tilde{\zeta}^{\BT_d}_t|=1|\tilde\zeta_t^{\BT_d}\ne\textbf{0},\forall t)=\tilde P^{\zeta_A}(\tilde{\zeta}^{\BT_d}_t=[\{o\}]|\tilde\zeta_t^{\BT_d}\ne\textbf{0},\forall t)\to 0.
$$
 For $k\ge 2$, it is easy to verify that  Proposition 2.6 in \cite{handjani1999complete} also holds for the threshold-one voter model on $\BT_d$.  

\end{proof}

Since ${\zeta}^{\BT_d}_t$ has no independent death rate, the method of the proof of Proposition \ref{lem57} and \ref{lem7} is no longer applicable. However, inspired by Proposition 3.2 in \cite{handjani1999complete}, we can still prove the following proposition when $\beta(d)\ne\frac{1}{\sqrt{d}}$.

\begin{proposition} \label{prop71}
If $\beta(d)\ne\frac{1}{\sqrt{d}}$, for any initial distribution $\mu$ satisfying $\mu(\zeta=\textbf{0}\text{ or }\textbf{1})=0$, 
$\mu$ is translation invariant or $\mu=\delta_\zeta$ where $\zeta$ is doubly dense defined in \eqref{dense},
we have
$$
\lim_{t\to\infty}\delta_{\zeta} U(t)= \mu_{1/2}^d,
$$
where $U(t)$ is the Markov semigroup of ${\zeta}^{\BT_d}_t$ and $\mu_{1/2}^d$ is the limiting distribution of the product measure with parameter 1/2.
\end{proposition}

\begin{proof}Let 
  \begin{center}$\Omega_{\infty}=\{\tilde{\zeta}^{\BT_d}_t\ne\textbf{0}$ for all $t $$\}$.
  \end{center}

By annihilating duality \eqref{2_cd}, it suffices to prove that 
$$P^{\mu,\zeta_A}( \zeta_{2}^{\BT_d}\cap\tilde{\zeta}^{\BT_d}_t\text{ is odd })\to\frac{1}{2}\tilde{P}^{\zeta_A}(\Omega_\infty)$$
when $\tilde{P}^{\zeta_A}(\Omega_\infty)>0$. Note that by the transition rates of $\tilde{\zeta}^{\zeta_A,\BT_d}_t$,  $|\tilde{\zeta}^{\zeta_A,\BT_d}_t|$ is always odd when $|A|$ is odd, so that $\tilde{P}^{\zeta_A}(\Omega_\infty)=1$ when $|A|$ is odd. 

Define
\begin{center}$
\bar{\zeta}^{\BT_d}_1=\{x\in \BT_d:\zeta^{\BT_d}_1(x)=1, \zeta^{\BT_d}_1(\overleftarrow{x})=0$ or $\zeta^{\BT_d}_1(x)=0, \zeta^{\BT_d}_1(\overleftarrow{x})=1$\}\end{center}where $\overleftarrow{x}$ is the the parent of $x$.

 Since $|\tilde{\zeta}_t^{\BT_d}|\to\infty$ in probability on $\Omega_{\infty}$ by Lemma \ref{lem73}, applying the proof of \eqref{590} in Proposition \ref{lem57}, in order to prove
\begin{center}$|\bar{\zeta}_1^{\BT_d}\cap\tilde{\zeta}_t^{\BT_d}|\to\infty$ in probability under $P^{\mu,\zeta_A}$ on $\Omega_{\infty}$,
\end{center}
we only need to show that when $\mu=\delta_\zeta$ for some doubly dense configuration $\zeta$, we have
$$\inf\limits_xP^\zeta(\bar\zeta_1^{\BT_d}(x)=1)>0.$$

Let $\CI_x=\{\eta:\exists y, z \in B_N(x),\zeta(y)=0,\zeta(z)=1\}$, 
 \bea\label{7111}\inf\limits_xP^\zeta(\bar\zeta_1^{\BT_d}(x)=1)
 &\geq\inf\limits_{x}\inf\limits_{\zeta\in\CI_x}P^\zeta(\bar\zeta_1^{\BT_d}(x)=1)
 \\&=\inf\limits_{\zeta\in\CI_o}P^\zeta(\bar\zeta_1^{\BT_d}(o)=1)\quad\text{ by translation invariance}
 \\&=\inf\limits_{\zeta\in\CI_o}P^{\zeta_{\{o,\overleftarrow{o}\}}}(\tilde{\zeta}^{\BT_d}_1\cap\zeta\text{ is odd})\quad\text{ by duality}
 \\&\geq \inf\limits_{\zeta\in\CI_o}P^{\zeta_{\{o,\overleftarrow{o}\}}}(\tilde{\zeta}^{\BT_d}_1\subset B_N(o)\text{ and }\tilde{\zeta}^{\BT_d}_1 \cap\zeta\text{ is odd})
\\&=\inf\limits_{\zeta\subseteq B_N(o),\zeta\in\CI_o}P^{\zeta_{\{o,\overleftarrow{o}\}}}(\tilde{\zeta}^{\BT_d}_1\subseteq B_N(o),\tilde{\zeta}^{\BT_d}_1 \cap\zeta\text{ is odd})
\\&=p>0\eea 
where the last equality holds because that there are finite configurations satisfying $\zeta\subseteq B_N(o)$ and each has a subset of odd number of vertices.

Let \begin{center}$U_t=\bar{\zeta}_1^{\BT_d}\cap\tilde{\zeta}_t^{\BT_d},
V_t=\{x:x\in U_t, \text{ $x$ is isolated}\}$.\end{center}
where $x$ is isolated if and only if $T_n^{y,S_y}\cap [1,2]=\emptyset$ for any $S_y\cap\{x,\overleftarrow{x}\}\ne\emptyset$, or $(y,S_y)=(x,S_x)$ for any $S_x$, or $(y,S_y)=( \overleftarrow{x},S_{\overleftarrow{x}})$ for any $S_{\overleftarrow{x}}\ne\{x\}$. 
It makes sure that
\bea&\text{$\forall x\in V_t $, if ${\zeta}^{\BT_d}_1(x)=1$, then $\zeta^{\BT_d}_2(x)=0$ iff $\text{there is no arrow from} \overleftarrow{x}\text{ to $x$ at [1,2]}$,}
\\&\text{$\forall x\in V_t $, if ${\zeta}^{\BT_d}_1(x)=0$, then $\zeta^{\BT_d}_2(x)=1$ iff $\text{there is at least an arrow from} \overleftarrow{x}\text{ to $x$ at [1,2]}$,}\eea
 
so that \bea\label{73}
|\bar{\zeta}_{2}^{\BT_d}\cap\tilde{\zeta}^{\BT_d}_t\cap V_t|=&\sum\limits_{x\in V_t,\zeta_1^{\BT_d}(x)=1}1_{\{\text{there is no arrow from} \overleftarrow{x}\text{ to $x$ at [1,2]}\}}\\&+\sum\limits_{x\in V_t,\zeta_1^{\BT_d}(x)=0}1_{\{\text{there is at least an arrow from} \overleftarrow{x}\text{ to $x$ at [1,2]}\}}.
\eea

Let 
\bea\label{72}
&g_x=1_{\{{\zeta}^{\BT_d}_1(x)=1,\text{there is no arrow from} \overleftarrow{x}\text{ to $x$ at [1,2]}\}}+1_{\{{\zeta}^{\BT_d}_1(x)=0,\text{there is at least an arrow from } \overleftarrow{x}\text{ to $x$ at [1,2]}\}},
\\&h=1-\{|\bar{\zeta}_{2}^{\BT_d}\cap\tilde{\zeta}^{\BT_d}_t\cap V^c_t| \text{ mod 2}\},
\\&\CG_t=\sigma(V_t,h,\zeta_1^{\BT_d}).\eea

Note that $\{x\in V_t\}_{x\in U_t}$ are 2-dependent, thus $|V_t|\to \infty$ in probability under $P^{\mu,\zeta_A}$ on $\Omega_\infty$.
Since given $ \CG_t$, $V_t,\zeta_1^{\BT_d}$ and $h$ are constant and $\{g_x, x\in V_t\}$ are mutually independent and independent of $ \CG_t$, by \eqref{73},  \eqref{72}, Lemma 2.3 of \cite{ bramson1991annihilating}, 
\begin{center}$|P^{\mu,\zeta_A}( \bar{\zeta}_{2}^{\BT_d}\cap\tilde{\zeta}^{\BT_d}_t\text{ is odd } |\CG_t)-\frac{1}{2}|=|P^{\mu,\zeta_A}(\sum\limits_{x\in V_t}g_x=h \text{  mod 2}|\CG_t)-\frac{1}{2}|\leq(1-2e^{-1})^{|V_t|}\to 0$ $\quad$ \end{center}in probability on $\Omega_{\infty}$.\\
    Taking expectation and by bounded convergence theorem,\begin{align*}P^{\mu}( \zeta_{2+t}\cap\zeta_A\text{ is odd } |\Omega_{\infty})=P^{\mu,\zeta_A}( \zeta_{2}^{\BT_d}\cap\tilde{\zeta}^{\BT_d}_t\text{ is odd } |\Omega_{\infty})\to\frac{1}{2}.
 \end{align*}

 \end{proof}

Inspired by the proof of Proposition 3.5 of \cite{handjani1999complete}, we will use a comparison of $\zeta^{\BT_d}_{t}$ and $\eta^{\varepsilon,\BT_d}_{t}$ for some $\delta>0$ to prove the complete convergence theorem for $\zeta^{\BT_d}_{t}$ when it survives strongly.

{\bf Proof of Theorem \ref{thm72}:}
\begin{proof}
When $\beta(d)<\frac{1}{\sqrt{d}}$, the result follows from Proposition \ref{prop71}, so that we only need to prove the result when $\zeta^{\BT_d}_{t}$ survives strongly.
Since 0 and 1 are equivalent, we may as well assume that $\sum_{x}[1-\zeta(x)]=\infty$. Then it suffices to prove that for any finite $B,\zeta\ne\textbf{0}$, 
$$
P^\zeta(\zeta^{\BT_d}_{t}\cap B\text{ is odd }|\tau_{\textbf{0}}=\infty)\to\frac{1}{2}\mu(\zeta:\zeta\cap\zeta_B\text{ is odd })=\frac{1}{2}\tilde P^{\zeta_B}(\Omega_{\infty} )
$$ when $\tilde P^{\zeta_B}(\Omega_{\infty} )>0$.

From the argument of Theorem 4.65 (c) of \cite{liggett1999stochastic}, we can prove that
\bea\label{70}
\text{$\eta^{\varepsilon,\BT_d}_t$ does not survive strongly when }\varepsilon=\varepsilon_2(d).
\eea   
 It follows that $\varepsilon_2(d)>0$, therefore there exists $\varepsilon>0$ such that $\eta^{\varepsilon,\BT_d}_t$ survives strongly so that we have the domination $\eta^{\varepsilon,\BT_d}_t\le\zeta^{\BT_d}_{t}\le\hat{\zeta}^{\varepsilon,\BT_d}_t$, where $\hat{\zeta}^{\varepsilon,\BT_d}_t$ denotes the threshold-one voter model with spontaneous birth rate $\varepsilon$. 
 By Lemma \ref{lem44} and Remark \ref{rmk62}, for any $\epsilon>0$, choose $N$ such that for any finite subset $D $ satisfying $|D|\geq N$, 
\bea\label{D}
&P^{\zeta_D}(\zeta^{\BT_d}_t\ne\emptyset \text{ for all t })>1-\epsilon ,
\\&P^{\zeta_D}(\eta^{\varepsilon,\BT_d}_t\ne\emptyset \text{ for all t })>1-\ep.
\eea 
It is easy to verify that $|\zeta^{\BT_d}_t|\to\infty$ a.s. on $\{\tau_{\textbf{0}}=\infty\}$, thus we can choose $U$ such that 
\bea\label{90}
P^{\zeta}(|\zeta^{\BT_d}_t|\geq N|\tau_{\textbf{0}}=\infty)>1-\epsilon
\eea for any $t\geq U$.

 By Theorem \ref{thm62}, there exists translation invariant measures $\nu_a$ and $\nu_b$ such that

\bea\label{7_1}
&\delta_{\zeta_D}S^{\varepsilon}(t)\to\ep'\nu_b+[1-\ep'] \delta_{\textbf{0}},
\\&\delta_{\zeta_D}\hat S^{\varepsilon}(t)\to \nu_a,
\eea
where $\ep'=P^{\zeta_D}(\eta^{\varepsilon,\BT_d}_t \text{ dies out })<\ep$.

Let 
\bea\label{7_2}\tilde{\nu}_b=\epsilon\delta_0+(1-\epsilon)\nu_b.\eea

 By Proposition \ref{prop71},
 $$\tilde{\nu}_bU(t)\to\epsilon\delta_0+(1-\epsilon)\mu_{1/2}^d, \nu_aU(t)\to\mu_{1/2}^d.$$
Observe that
$$1_{\{\zeta^{\BT_d}_t\cap B\text{ is odd }\}}=\frac{1}{2}[1-\prod_{x\in B}[1-2\zeta^{\BT_d}_t(x)]],$$ which is a linear combination of at most $2^{|B|}$ increasing functions of the form $\prod_{x\in C}\zeta^{\BT_d}_t(x),C\subseteq B$, so that for any sufficient large $T$, 
\begin{center}$|P^{\rho}(\zeta^{\BT_d}_T\cap B\text{ is odd })-\frac{1}{2}\tilde P^{\zeta_B}(\Omega_{\infty} )|<2^{|B|+1}\epsilon$\end{center} uniformly for any $\tilde{\nu}_b\le\rho\leq\nu_a$.

Then  for any $D$ satisfying $|D|\ge N$, for any sequence $\{ t_n\}$, since the space of all probability measures is compact, there exists $\{r_n\}\subset\{t_n-T-U\}$ such that $\zeta_DS(r_n)\to\rho$. Then by \eqref{7_1} and \eqref{7_2}, $\tilde{\nu}_b\leq\rho\leq\nu_a$. 
It follows that
\bea\label{91}
\lim\limits_{n\to\infty}|P^{\zeta_D}(\zeta^{\BT_d}_{r_n+T}\cap B\text{ is odd })-\frac{1}{2}\tilde P^{\zeta_B}(\Omega_{\infty} )|= |P^{\rho}(\zeta^{\BT_d}_T\cap B\text{ is odd })-\frac{1}{2}\tilde P^{\zeta_B}(\Omega_{\infty} )|<2^{|B|+1}\epsilon.\eea

By \eqref{D} and \eqref{91},
\bea\label{haha}\limsup\limits_{n\to\infty}|P^{\zeta_D}(\zeta^{\BT_d}_{r_n+T}\cap B\text{ is odd }|\tau_{\textbf{0}}=\infty)-\frac{1}{2}\tilde P^{\zeta_B}(\Omega_{\infty} )|<2^{|B|+2}\epsilon.\eea
By \eqref{90} and \eqref{haha},
\begin{center}$\limsup\limits_{n\to\infty}|P^\zeta(\zeta^{\BT_d}_{r_n+T+U}\cap B\text{ is odd }|\tau_{\textbf{0}}=\infty)-\frac{1}{2}\tilde P^{\zeta_B}(\Omega_{\infty} )|<2^{|B|+3}\epsilon.$\end{center}
Then we can get the desired result by letting $\epsilon\to 0$.
\end{proof}

\section{Appendix}
\subsection{Appendix (A)}\label{Appendix (A)}
\label{appendix}
{\bf Proof of Theorem \ref{survive 2}:}
 Recall the definition
\bea
\label{test_function_1}
 f(A)=\lim\limits_{\epsilon\to 0}\frac{1-\nu_{\epsilon}(\eta=0 \text{ on A)}}{\nu_{\epsilon}(\eta(x)=1)}
\eea
where $\nu_{\epsilon}$ is the translation invariant measure on $\BT_2$ satisfying the following conditions: 
\begin{itemize}
\item$\nu_{\epsilon}(\eta(x)=1)=\frac{\epsilon}{q+\epsilon}$.
\item $\nu_{\epsilon}(\eta(x)=0|\eta(y)=1)=q$, $\nu_{\epsilon}(\eta(x)=0|\eta(y)=0)=1-\epsilon$ when $x\sim y$.
\item Conditional on $\eta(x)$, $\eta(x_{(1)})$, $\eta(x_{(2)})$, $\eta(x_{(3)})$ are independent with respect to $\nu_{\epsilon}$, where $x_{(i)},1\le i\le 3$ are the neighbors of $x$.
\end{itemize}
 One may explicitly construct such measure by defining a 0-1 random variable on $o$ according to its marginal and then recursively define them on $\partial B_n(o)$ in each step.

The existence of $f(A)$ is easy to prove.
For finite and nonempty $A$, denote $y_{A}$ as a site in $A$ with the minimum generation. For a configuration $\psi\in \{0,1\}^A$, define subset $D_{\psi}=\{y:\psi(y)=1\}$.

Note that $\tilde{\xi}^{\lambda,\BT_d}_t$ is the Markov process on the countable space $Y$ now, where $Y$ is the set of all finite subsets of $\BT_d$, thus the derivation of $E^{\xi_A}f(\tilde{\xi}^{\lambda,\BT_d}_t)$ is applicable even though $f$ is discontinuous in $\{0,1\}^{\BT_d}$ on the ordinary topology. 

Denote by $\CB$ the set of the components of $A$, 
 \bea\label{430}
 h(A)&=\frac{d}{dt}E^{\xi_A}f(\tilde{\xi}^{\lambda,\BT_d}_t)|_{t=0}
 \\&=\sum\limits_{B}q(A,B)[f(B)-f(A)]
 \\&=\lambda\sum\limits_{x\in A}\lim\limits_{\epsilon\to 0}\frac{\nu_{\epsilon}(\eta=0 \text{ on A}, \exists x*\sim x \text{ s.t.} \eta(x*)=1)}{\nu_{\epsilon}(\eta(x)=1)}-\sum\limits_{x\in A}\nu(\eta=0 \text{ on A}\backslash\{x\}|\eta(x)=1)
 \\&=\sum\limits_{B\in\CB}\sum\limits_{x\in B}\left[\lambda\sum\limits_{ \text{
\tiny $\begin{array}{c}
\psi\in\{0,1\}^{\CN_x}\\
D_{\psi}\ne\emptyset\end{array}$
}}\nu(\eta=0 \text{ on } A,\eta=\psi \text{ on }\CN_x|\eta(y_{D_\psi})=1)
\right] \sum\limits_{x\in A}\nu(\eta=0 \text{ on A}\backslash\{x\}|\eta(x)=1)
\\&\overset{\vartriangle}{=} \sum\limits_{B\in\CB}h_B(A)- \sum\limits_{x\in A}\nu(\eta=0 \text{ on A}\backslash\{x\}|\eta(x)=1)
 \eea
where $h_B(A)$ can be seen as the contributions of $B$. 

Next we want to show that \eqref{430} has the same form as (4.2) in \cite{liggett1996multiple}. Since there is not much difference in the case when $|B|=1$ or $|B|>1$, we just prove the case $|B|=1$ and the other case follows from the same strategy.

When $B=\{x\}$ is a singon, $\CN_x\cap A=\{x\}$. Define
\begin{align*}\gamma(y)=1-\nu(\eta=0\text{ on } A\cap S_y(x)|\eta(y)=1)\end{align*}
where $S_y(x)$ is the conected component of $\BT_d\backslash\{x\}$ containing $y$. Then\bea \label{43}
h_B(A)=&\lambda[\nu(\eta=0 \text{ on  }A,\eta(x_{(1)})=1,\eta(x_{(2)})=0,\eta(x_{(3)})=0|\eta(x_{(1)})=1)
\\&+\nu(\eta=0 \text{ on } A,\eta(x_{(1)})=1,\eta(x_{(2)})=1,\eta(x_{(3)})=0|\eta(x_{(1)})=1)\\&
+\nu(\eta=0 \text{ on  }A,\eta(x_{(1)})=1,\eta(x_{(2)})=1,\eta(x_{(3)})=1|\eta(x_{(1)})=1)\\&
+\nu(\eta=0 \text{ on  }A,\eta(x_{(1)})=1,\eta(x_{(2)})=0,\eta(x_{(3)})=1|\eta(x_{(1)})=1)\\&
+\nu(\eta=0 \text{ on  }A,\eta(x_{(1)})=0,\eta(x_{(2)})=1,\eta(x_{(3)})=0|\eta(x_{(2)})=1)\\&
+\nu(\eta=0 \text{ on  }A,\eta(x_{(1)})=0,\eta(x_{(2)})=1,\eta(x_{(3)})=1|\eta(x_{(2)})=1)\\&
+\nu(\eta=0 \text{ on  }A,\eta(x_{(1)})=0,\eta(x_{(2)})=0,\eta(x_{(3)})=1|\eta(x_{(3)})=1)]
\\&\overset{\vartriangle}{=}\lambda\times[(1)+(2)+(3)+(4)+(5)+(6)+(7)].
\eea
Since the proof for (1), (5) and (7) are similar, we just take (1) for an example.
 \begin{align*}&\nu(\eta=0 \text{ on A},\eta(x_{(1)})=1,\eta(x_{(2)})=0,\eta(x_{(3)})=0|\eta(x_{(1)})=1)\\&=\nu(\eta=0 \text{ on A}\cap[ \BT_2\backslash S_{x_{(1)}}(x)],\eta(x_{(2)})=0,\eta(x_{(3)})=0|\eta(x)=0)\nu(\eta(x)=0|\eta(x_{(1)})=1)\\&\times\nu(\eta=0 \text{ on A}\cap S_{x_{(1)}}(x)|\eta(x_{(1)})=1)
 \\&=q\nu(\eta=0 \text{ on A}\cap S_{x_{(1)}}(x)|\eta(x_{(1)})=1)
\\&=q(1-\gamma(x_{(1)})).
 \end{align*}
Similarly, we take (2) for an example to prove that (2), (3), (4) and (6) all equal 0.
\begin{align*}&\nu(\eta=0 \text{ on A},\eta(x_{(1)})=1),\eta(x_{(2)})=1,\eta(x_{(3)})=0|\eta(x_{(1)})=1)\\&=\nu(\eta=0 \text{ on A}\cap( T_2\backslash S_{x_{(1)}}(x)),\eta(x_{(2)})=1,\eta(x_{(3)})=0|\eta(x)=0)\nu(\eta(x)=0|\eta(x_{(1)})=1)\\&\times\nu(\eta=0 \text{ on A}\cap S_{x_{(1)}}(x)|\eta(x_{(1)})=1)
 \\&\leq \nu(\eta(x_{(2)})=1|\eta(x)=0)\\&=0.
 \end{align*}

Thus
\bea\label{48}
h_B(A)=\lambda q\sum\limits_{i=1}^{3}[1-\gamma(x_{(i)})]-\prod_{i=1}^{3}[1-(1-q)\gamma(x_{(i)})],
\eea

which is the same as (4..2) in \cite{liggett1996multiple} such that
\bea\label{417}h(A)= \lambda\sum\limits_{x\in A,y\not\in A,x\sim y}\nu(\eta=0 \text{ on }A|\eta(y)=1)-\sum\limits_{x\in A}\nu(\eta= 0 \text{ on } A|\eta(x)=1)
\eea
 So that by the argument in page 1693 and 1695 in \cite{liggett1996multiple}, $h(A)\geq 0$ when $\lambda\geq 0.6369$. Specifically, 
\bea \sum\limits_{x\in A,y\not\in A,x\sim y}\nu(\eta=0 \text{ on }A|\eta(y)=1)\geq\frac{1}{0.6369}\sum\limits_{x\in A}\nu(\eta= 0 \text{ on } A|\eta(x)=1).\label{410}
\eea

Next we want to get a nonnegative supermartingale. First we have the following lemma:
\begin{lemma}[Page1690 of \cite{liggett1996multiple}] 
\label{fact}
Suppose $\xi_t$ is a pure jump process on a countable set with transition rates $q(A,B)$.  If $f$ satisfies the following conditions for some $0<C, M<\infty$:\\
(i) $ \forall A,\sum\limits_B q(A,B)[f(B)-f(A)]\geq Cf(A)$.\\
(ii) $|f(B)-f(A)|\leq$ M whenever $q(A,B)>0$.\\
(iii) $q(A)\leq Cf(A),\forall A$.\\
Then $f(\xi_{t\land\tau})^{-1}$ is a supermartingale where $\tau=\inf\{t,f(\xi_t)\le c\}$ for some sufficiently large enough $c$.
\end{lemma}

 Now we only need to show that the test  function \eqref{test_function_1} satisfies the conditions in Lemma \ref{fact}. For any $x\in A$, 
\bea
1&\geq f(A)-f(A\backslash\{x\})
\\&=\prod_{x_{(i)}}\nu(\eta=0 \text{ on }A\cap S_{x_{(i)}}(x)|\eta(x)=1)
\\&\ge\prod_{x_{(i)}}\nu(\eta=0 \text{ on }A\cap S_{x_{(i)}}(x),\eta(x(i))=0|\eta(x)=1)
\\&= q^3\label{412}.
\eea

 For any $x\in A$ satisfies $\CN_x\cap A\ne\emptyset$, assume $x_{(1)}\not\in A$, by \eqref{412}
\bea7\geq f(A\cup\CN_x)-f(A)
&=\sum\limits_{ \text{
\tiny $\begin{array}{c}
\psi\in\{0,1\}^{\CN_x}\\
D_{\psi}\ne\emptyset\end{array}$
}}\nu(\eta=0 \text{ on } A,\eta=\psi \text{ on }\CN_x|\eta(y_{D_\psi})=1)
\\&\geq\nu(\eta=0\text{ on }A\cup\{x_{(2)},x_{(3)}\}|\eta(x_{(1)})=1)
\\&\geq q^3.\label{411}
\eea
It follows that 
\bea C_1|A|\leq f(A)\leq C_2|A|.\label{413}\eea
When $\lambda>0.6369$,
 \bea&\sum\limits_B q(A,B)[f(B)-f(A)]
\\&=(\lambda-0.6369)\sum\limits_{x\in A,y\not\in A,x\sim y}\nu(\eta=0 \text{ on }A|\eta(y)=1)
\\&+\left[0.6369\sum\limits_{x\in A,y\not\in A,x\sim y}\nu(\eta=0 \text{ on }A|\eta(y)=1)
-\sum\limits_{x\in A}\nu(\eta=0\text{ on } A\backslash\{x\}|\eta(x)=1)\right]\quad\text{ by \eqref{417}}\\&\geq\frac{\lambda-0.6369}{0.6369}\sum\limits_{x\in A}\nu(\eta=0\text{ on } A\backslash\{x\}|\eta(x)=1) \quad\text{ by \eqref{410}}
\\&\geq C_3|A|\quad \text{by \eqref{412}}
\\&\geq C_4f(A).\quad \text{by \eqref{413}}\label{414}
\eea
When $q(A,B)>0$, by \eqref{411} and \eqref{412}, 
  \bea|f(A)-f(B)|\leq 7\label{415}\eea
 and since $q(A)\leq C_5|A|$, it follows from \eqref{413} that
  \bea 
  q(A) \leq C_6f(A)\label{416}
  \eea
where $C_i,1\leq i\leq 6$ are positive and finite constants. Then \eqref{414}, \eqref{415} and \eqref{416} gives that $\tilde{\xi}^{\lambda,A,\BT_d}_t$ satisfies the conditions of Lemma \ref{fact}, thus $f(\xi^{\lambda,A,\BT_d}_{t\land\tau})^{-1}$ is a nonnegative supermartingale. By the argument in Section \ref{sec_32}, we complete the proof. \qed

\subsection{Appendix (B)}\label{Appendix (B)}
In order to prove Lemma \ref{le64}, we need the following lemma, which can be proved easily by strong Markov property, we omit the proof here.
 \begin{lemma}\label{le63} 
For any $\lambda,d$,

\bea\label{641}\lim\limits_{n\to\infty}[\sup\limits_t P^o(x_n\in\xi^{\lambda,\hat{\BT}^d_{x_n}}_t)]^{1/n}=\beta(\lambda,d).
\eea
\end{lemma}

{\bf Proof of Lemma \ref{le64}:}
By \eqref{641}, for all $a$ satisfies $ \frac{1}{\sqrt{d}}<a<\beta(\lambda,d)$, $n$ large enough, there exists $ t$ such that
\bea \label{00}P^o(x_n\in\xi^{\lambda,\hat{\BT}^d_{x_n}}_t)> a^n.
\eea
Let
$$B_0=\{o\},B_i=\cup_{x\in B_{i-1}}\{y\in \xi_t^{\lambda,\BT^d_x},|y-x|=n\}$$
 where $\xi^{\lambda,\BT^d_x}_{\cdot}$ are independent copies of $\xi^{\lambda,\BT^d_o}_\cdot$. Then $|B_j|$ is a supercritical branching process since $E|B_1|>(ad)^n>1$ by the domination of $\xi^{\lambda,\BT^d_o}_{jt}$. Let $b=[P^o(x_n\in\xi^{\lambda,\hat{\BT}^d_{x_n}}_t)]^{1/n}\geq a$, by the limit theorem for the supercritical branching processes
\begin{center}$\lim\limits_{j\to\infty}\frac{|B_j|}{(d^nb^n)^j}=X$\end{center}exists a.s. Thus by Theorem 5.3.10 of \cite{durrett2010probability}, $X\not\equiv 0$ so that there exists $\epsilon>0$ such that 
\bea \label{642}P( |B_j|\geq\epsilon(da)^{nj})\geq\epsilon\eea
 for all sufficiently large $j$.

Now let \begin{center}$r_i=P^o(o\in\xi^{\lambda,\BT_o^d}_{2ijt})$.\end{center} 
Note that by strong Markov property, 
\bea\label{01}
P^{x_{jn}}(o\in\xi^{\lambda,\BT_o^d}_{jt})\geq \left[P^{o}(x_n\in\xi^{\lambda,\hat{\BT}_{x_n}^d}_{t})\right]^j\geq a^{nj}.
\eea
By \eqref{01}, strong Markov property and monotonicity,
\bea
r_{i+1}&\geq P^o(\exists |x|=jn, x\in\xi^{\lambda,\BT_o^d}_{(2i+1)jt})P^{x_{jn}}(o\in\xi^{\lambda,\BT_o^d}_{jt})\\
&\ge P^o(\exists |x|=jn, x\in\xi^{\lambda,\BT_o^d}_{(2i+1)jt})a^{jn}\\
&\geq\epsilon a^{jn}[1-(1-r_i)^{\lfloor\epsilon (ad)^{jn}\rfloor}]\quad\text{ by \eqref{642}}\\
&\overset{\vartriangle}{=}f(r_i). 
\eea
Note that $f$ is an increasing function with $f(0)=0$, $f(1)=\epsilon a^{nj}<1$, $f'(0)\geq\epsilon\lfloor\epsilon (a^2d)^{nj}\rfloor>1$ for some $ j$ sufficiently large. So there exists $0<r^*<1$ such that $f(r^*)=r^*$. Since $r_0=1>r^*$, then for any $i$, we have $r_{i+1}\geq f(r_i)\geq f(r^*)=r^*$. To extend to all times, simply 
use\begin{center}$P^o(o\in\xi^{\lambda,\BT_d}_s)\geq e^{-2jt}r_i$ $\quad 2ijt\leq s< 2(i+1)jt$.\end{center}\qed

\bibliography{complete_convergence}
\bibliographystyle{plain}

\end{document}